\documentclass[12pt,a4paper,oneside,reqno]{article}
\usepackage[utf8]{inputenc}
\usepackage[english]{babel}
\usepackage[symbol]{footmisc}
\usepackage{amssymb,amsmath,amsthm,amsfonts,xcolor,enumerate,hyperref, comment,longtable, cleveref}
\usepackage{verbatim}
\usepackage{times}
\usepackage{cite}
\usepackage{pdflscape}
\usepackage[mathcal]{euscript}
\usepackage{tikz}
\usepackage{tikz-cd}
\usepackage{hyperref}
 
\usepackage{cancel}
\usepackage{stmaryrd}
 
\usepackage{amsfonts}
\usepackage{amssymb}
\usepackage{times}
\usepackage{xcolor}
\usepackage{tkz-graph}
\usepackage{url}
\usepackage{float}
\usepackage{tasks}
\usepackage{array}
\newcolumntype{C}[1]{>{\centering\arraybackslash}p{#1}}
\newcolumntype{L}[1]{>{\raggedright\arraybackslash}p{#1}}

\usepackage{cite}
\usepackage{hyperref}

 \usepackage{fancyhdr} 
\usepackage{amsfonts}
\usepackage{amsmath}
\usepackage{eurosym}
\usepackage{geometry}

\usepackage{caption,booktabs}
\usepackage{fouriernc}
\DeclareMathAlphabet{\mathcal}{OMS}{cmsy}{m}{n}
\theoremstyle{plain}
\newtheorem{theorem}{Theorem}[section]
\newtheorem{lemma}[theorem]{Lemma}
\newtheorem{proposition}[theorem]{Proposition}
\newtheorem{corollary}[theorem]{Corollary}

\theoremstyle{definition}
\newtheorem{definition}[theorem]{Definition}
\newtheorem{example}[theorem]{Example}
\newtheorem{remark}[theorem]{Remark}

\numberwithin{equation}{section}

\newcommand{\C}{\mathbb{C}}
\DeclareMathOperator{\Ann}{Ann}
\DeclareMathOperator{\Span}{span}
\begin{document}

\noindent{\Large
\begin{center}
    $\delta$-Mock-Novikov Algebras: Structural Theory, Operad and Poisson-Type Constructions
\end{center}}
 
\bigskip

\begin{center}
{\bf
   Jie Ruan\footnote{School of Mathematics and Statistics, Northeast Normal University, Changchun, 130024, China; ruanriss@126.com} \ \  
   Zafar Normatov\footnote{School of Mathematics, Jilin University, Changchun, China; z.normatov@inbox.ru}}
\end{center}

\noindent {\bf Abstract.}
We introduce $\delta$-mock-Novikov algebras as a parameter-dependent mock analogue of Novikov algebras. For $\delta=1$, the commutator of
a mock-Novikov algebra defines a Malcev algebra.  Every
$\delta$-mock-Novikov algebra satisfying
$\mathcal{A}^2\subseteq\operatorname{Ann}(\mathcal{A})$ is shown to be
differentially special. At the operadic level, the binary quadratic operad governing
$\delta$-mock-Novikov algebras is quadratically
self-dual but not Koszul. We prove that finite-dimensional $\delta$-mock-Novikov algebras are
nilpotent in characteristic zero and classify those of dimension at
most four. Finally, we study the associated Poisson-type structures and establish
relations and constructions among them via polarization,
depolarization, and tensor products.
\bigskip 


\bigskip 
 

\medskip


\tableofcontents 
\section{Introduction}
Novikov algebras arose in the study of Hamiltonian operators in the
formal calculus of variations and Poisson brackets of hydrodynamic
type \cite{GelDorf79,DubNov83,bn85}. Their defining identities encode
the Jacobi identity for the corresponding class of homogeneous
first-order Hamiltonian operators. 
Algebraically, Novikov algebras form a distinguished subclass of
pre-Lie algebras and are therefore Lie-admissible
\cite{Burde06}.

On the mock side, replacing skew-symmetry by commutativity while
retaining the Jacobi identity leads to mock-Lie, or Jacobi-Jordan,
algebras \cite{zus2,BF}. This replacement is not merely formal:
although mock-Lie algebras satisfy the Jacobi identity, classical
results such as the Ado and PBW theorems do not extend to them in
general \cite{zus2}. Mock-pre-Lie algebras similarly arise as
counterparts of pre-Lie algebras in the mock setting. Since Novikov
algebras form a distinguished subclass of pre-Lie algebras, one is
naturally led to the parallel picture
\[
\begin{array}{ccc}
\text{Lie} & \longleftarrow & \text{pre-Lie} \supset \text{Novikov},\\
\text{mock-Lie} & \longleftarrow & \text{mock-pre-Lie} \supset ?
\end{array}
\]
This raises the question of which class of algebras should occupy the
missing position and provides one of the motivations for introducing
$\delta$-mock-Novikov algebras.

A second source of motivation comes from the differential
realization of Novikov algebras. Every Novikov algebra over
$\mathbb C$ embeds into one arising from a commutative associative
algebra equipped with a derivation
\cite{ZhangChenBokut2019}. In the classical operadic picture, the
commutative and Lie operads are quadratically dual, while on the
mock side the quadratic dual of the mock-Lie operad is the operad
of anti-commutative anti-associative algebras
\cite{GetzlerKapranov1995,zus2}. This suggests using
anti-commutative anti-associative algebras in place of commutative
associative algebras in a mock differential construction.
Together with the theory of $\delta$-derivations
\cite{fil1} and the corresponding $\delta$-Novikov constructions
\cite{K}, this leads naturally to the class of
$\delta$-mock-Novikov algebras introduced below.

Accordingly, we introduce a parameter-dependent mock analogue of
Novikov algebras, called $\delta$-mock-Novikov algebras
(Definition~\ref{mN}). For $\delta=1$, this notion coincides with
the Novikov-Jacobi-Jordan algebras considered in~\cite{Jie}.
Moreover, after passing to the negative opposite multiplication,
mock-Novikov algebras are identified with the anti-Novikov algebras
studied in~\cite{ChtiouiMabroukMakhlouf2025}.  We also prove that the commutator of a
mock-Novikov algebra defines a Malcev algebra. Thus the
Novikov-Lie commutator construction has a Malcev analogue in the
mock setting.

The operadic behavior provides a first point of comparison with the
classical theory. Operad theory provides a natural framework for studying algebraic
structures defined by multilinear identities
\cite{May1972,LodayVallette2012}. In particular, Ginzburg and
Kapranov established the theory of quadratic duality for operads
\cite{GK94}. For Novikov algebras, Dzhumadil'daev proved that the
left and right Novikov operads are quadratically dual and that the
Novikov operad is not Koszul \cite{Dzhumadildaev2011}. For the mock-Lie operad, however, the situation is different: its
quadratic dual is the operad governing anti-commutative
anti-associative algebras \cite{zus2}. We show that the operad $\mathcal M_\delta$ governing
$\delta$-mock-Novikov algebras is quadratically self-dual, via the
opposite multiplication, but is not Koszul. Thus, from the operadic
viewpoint, $\delta$-mock-Novikov algebras behave more like classical
Novikov algebras than mock-Lie algebras.

The finite-dimensional structure reveals a different pattern.
Zel'manov
proved that every finite-dimensional simple Novikov algebra over an
algebraically closed field of characteristic zero is one-dimensional
\cite{ze87}. More generally, finite-dimensional Novikov algebras
need not be nilpotent. By contrast, finite-dimensional mock-Lie
algebras are nilpotent \cite{zus2}, and we prove that the same holds
for every finite-dimensional $\delta$-mock-Novikov algebra over a
field of characteristic zero. Thus, while the operadic behavior of
$\delta$-mock-Novikov algebras is closer to that of classical
Novikov algebras, their finite-dimensional nilpotency is closer to
the mock-Lie case. Over $\mathbb C$, this strong restriction further
leads to a complete classification in dimensions at most four.

A further connection with the classical theory arises from
Novikov-Poisson algebras. These algebras admit natural tensor-product
constructions \cite{Xu96,XuX}. Transposed Poisson algebras arise naturally from
$\frac12$-derivations of Lie algebras: their compatibility condition
expresses that multiplication by each element of the commutative
associative algebra acts as a $\frac12$-derivation of the Lie bracket
\cite{TP1}. Gelfand-Dorfman algebras provide another classical link
between Novikov and Lie structures; in particular, the commutator of
a Novikov algebra yields a Gelfand-Dorfman algebra
\cite{GelDorf79}, while commutative Novikov Gelfand--Dorfman
algebras are closely related to transposed Poisson algebras
\cite{SartayevTP}.

Guided by these classical constructions, we replace the commutative
associative and Lie structures by their mock counterparts and study
the resulting Poisson-type objects. This leads naturally to
$\delta$-mock-Novikov-Poisson, $\delta$-mock-Poisson, transposed
$\delta$-mock-Poisson, and $\delta$-mock-Gelfand-Dorfman algebras.
We establish the corresponding compatibility conditions and relate
these structures through polarization and depolarization, following
the one-operation viewpoint of Markl and Remm
\cite{MarklRemm}. We also obtain tensor-product constructions linking
the mock and classical settings.

The main objective of this paper is to develop the structural theory
of $\delta$-mock-Novikov algebras and to investigate their relationship
with the classical theory of Novikov algebras. Using anti-commutative anti-associative algebras endowed with
$\delta$-derivations, we construct
$\delta$-mock-Novikov algebras and formulate the corresponding notion
of differential speciality.  At the operadic level, we prove that the
binary quadratic operad governing $\delta$-mock-Novikov algebras is
quadratically self-dual but not Koszul. To
study the finite-dimensional case, we use left and right multiplication
operators together with Jacobson's weakly closed theorem to prove
nilpotency. By analyzing the square $\mathcal{A}^2$ in low dimensions,
we show that the classification in dimensions at most four reduces to
the known classification of complex two-step nilpotent algebras. We also develop the associated Poisson-type structures and establish their relationships through
polarization, depolarization, and tensor-product constructions.

The paper is organized as follows. Section~\ref{sec2} introduces
$\delta$-mock-Novikov algebras and develops their basic theory,
including the loop construction and differential realizations
(Theorem~\ref{prop:two-step-special}).
Section~\ref{sec3} studies quadratic duality and Koszulity of the
associated operad
(Theorems~\ref{thm:quadratic-dual-delta-mN}
and~\ref{thm:delta-mN-non-Koszul}).
Section~\ref{sec4} establishes the nilpotency of finite-dimensional
$\delta$-mock-Novikov algebras over fields of characteristic zero
(Theorem~\ref{thm:nilpotent}).
Section~\ref{sec5} gives the classification in dimensions at most four
(Theorem~\ref{thm:two-step}) and presents a five-dimensional example
with $\mathcal{A}^3\neq0$, showing that the conclusion
$\mathcal{A}^2\subseteq\operatorname{Ann}(\mathcal{A})$ valid in
dimensions at most four does not extend to dimension five.
Finally, Section~\ref{sec6} studies the associated Poisson-type
structures and the constructions and correspondences among them
(see diagram~\eqref{diag:mock-structures} and
Theorem~\ref{thm:tensor-unified}).

Throughout the paper, unless explicitly stated otherwise, the ground
field is $\mathbb C$ and $\delta\in\mathbb C^\times$. 
Whenever the relevant multiplication is clear from the context, we
suppress its symbol and write $xy$ for $x\circ y$ or $x\cdot y$, as
appropriate.

\section{The differential realizations of \texorpdfstring{$\delta$}{delta}-mock-Novikov algebras}\label{sec2}
This section develops the basic theory of $\delta$-mock-Novikov
algebras, including their left and right versions, symmetrization,
and loop construction. We then study differential realizations
arising from $\delta$-derivations of anti-commutative
anti-associative algebras and introduce the notion of differential
speciality.

Let $(\mathcal{A},\circ)$ be an algebra. For $a,b,c\in \mathcal{A}$, define the
\emph{$\delta$-anti-associator} by
\[
(a,b,c)_{+}^{\delta}
=\delta(a\circ b)\circ c+a\circ(b\circ c).
\]

\begin{definition}
An algebra $(\mathcal{A},\circ)$ is called a \emph{left
$\delta$-mock-pre-Lie algebra} if
$(a,b,c)_{+}^{\delta}=-(b,a,c)_{+}^{\delta}$ for all
$a,b,c\in \mathcal{A}$; equivalently,
\begin{equation}\label{deltaMN1}
\delta(a\circ b)\circ c+\delta(b\circ a)\circ c
+a\circ(b\circ c)+b\circ(a\circ c)=0.
\end{equation}
\end{definition}

For $\delta=1$, this reduces to the usual mock-pre-Lie
(pre-Jacobi-Jordan) identity \cite{bbmm}. Similarly, a \emph{right
$\delta$-mock-pre-Lie algebra} is defined by
$(a,b,c)_{+}^{\delta}=-(a,c,b)_{+}^{\delta}$. For $\delta=-1$ we have the definition of anti-mock-pre-Lie algebras introduced in \cite{AAA}.

\begin{definition}\label{mN}
A \emph{left $\delta$-mock-Novikov algebra} is a left
$\delta$-mock-pre-Lie algebra $(\mathcal{A},\circ)$ satisfying the
right-anti-commutative identity
\begin{equation}\label{deltaMN2}
(a\circ b)\circ c+(a\circ c)\circ b=0,
\qquad a,b,c\in \mathcal{A}.
\end{equation}
\end{definition}

The right version is defined analogously. Passing to the opposite
multiplication $a\ast b=b\circ a$ interchanges left
$\delta$-mock-Novikov algebras with right
$\delta^{-1}$-mock-Novikov algebras. Henceforth, we consider only the
left version and omit the qualifier \emph{left}.

Recall that an algebra $(\mathcal{A},\circ)$ is \emph{anti-commutative} if
$x\circ y=-y\circ x$, and \emph{anti-associative} if
$(x\circ y)\circ z=-x\circ(y\circ z)$
for all $x,y,z\in \mathcal{A}$.
The
anti-commutative case is particularly natural in the mock setting and
turns out to recover precisely the class of anti-commutative
anti-associative algebras.

\begin{proposition}\label{prop:anti-mN}
Let $(\mathcal{A},\circ)$ be an anti-commutative algebra. Then $(\mathcal{A},\circ)$ is a
$\delta$-mock-Novikov algebra if and only if it is anti-associative.
\end{proposition}
\begin{proof}
Assume first that $(\mathcal{A},\circ)$ is a $\delta$-mock-Novikov algebra. Since $\circ$ is anti-commutative, \eqref{deltaMN1} reduces to
$x\circ(y\circ z)+y\circ(x\circ z)=0$. Moreover, by \eqref{deltaMN2},
\[
(x\circ y)\circ z=-(x\circ z)\circ y
=y\circ(x\circ z)
=-x\circ(y\circ z),
\]
so $(\mathcal{A},\circ)$ is anti-associative.

Conversely, suppose that $(\mathcal{A},\circ)$ is anti-commutative and anti-associative. Then
\[
(x\circ y)\circ z+(x\circ z)\circ y
=-x\circ(y\circ z)-x\circ(z\circ y)=0,
\]
so \eqref{deltaMN2} holds. Also, anti-commutativity gives
$(x\circ y)\circ z+(y\circ x)\circ z=0$, while anti-associativity implies
$x\circ(y\circ z)+y\circ(x\circ z)=0$. Hence \eqref{deltaMN1} holds as well.
\end{proof}

The commutative case is highly restrictive: commutativity together
with the right-anti-commutative identity already forces all triple
products to vanish.

\begin{proposition}
Let $(\mathcal{A},\circ)$ be a commutative $\delta$-mock-Novikov algebra over a
field of characteristic different from $2$. Then
\[
\mathcal{A}^3=0.
\]
\end{proposition}

The classical commutator of a Novikov algebra defines a Lie algebra.
In the mock setting, the corresponding construction is given by
symmetrization.

\begin{proposition}\label{bdnov}
Let $(\mathcal{A},\circ)$ be a $\delta$-mock-Novikov algebra and define
\[
\{x,y\}=x\circ y+y\circ x.
\]
Then $(\mathcal{A},\{-,-\})$ is a mock-Lie algebra.
\end{proposition}

\begin{proof}
The multiplication $\{-,-\}$ is symmetric. By \eqref{deltaMN2},
\[
\{x,y\}\circ z+\{y,z\}\circ x+\{z,x\}\circ y=0.
\]
Summing \eqref{deltaMN1} cyclically and using the above
identity gives
$x\circ\{y,z\}+y\circ\{z,x\}+z\circ\{x,y\}=0$. Consequently,
\[
\{\{x,y\},z\}+\{\{y,z\},x\}+\{\{z,x\},y\}=0,
\]
which proves the assertion.
\end{proof}
Recall that an anti-commutative algebra
$(\mathcal M,[-,-])$ is called a \emph{Malcev algebra} if
\[
J(x,y,[x,z])=[J(x,y,z),x],
\]
where
\[
J(x,y,z)=[[x,y],z]+[[y,z],x]+[[z,x],y]
\]
is the Jacobian; see, for example,~\cite{Sagle1961}.
For mock-Novikov algebras, the commutator always satisfies this
identity. 
\begin{proposition}\label{prop:mN-Malcev}
Let $(\mathcal A,\circ)$ be a mock-Novikov algebra over a field of
characteristic zero. Define
$[x,y]=x\circ y-y\circ x.$
Then $(\mathcal A,[-,-])$ is a Malcev algebra.
\end{proposition}
\begin{proof}
Let $\mathcal F$ be the free nonassociative algebra on $x,y,z$, and
let $\mathcal F_{(2,1,1)}$ denote its multihomogeneous component
spanned by all nonassociative monomials in which $x$ occurs twice
and $y,z$ occur once. Since there are five association types and
twelve distinct permutations of $x,x,y,z$, we have
$\dim \mathcal F_{(2,1,1)}=60$.
Let
\[
F(a,b,c)=(a\circ b)\circ c+(a\circ c)\circ b
\]
and
\[
G(a,b,c)=(a\circ b)\circ c+(b\circ a)\circ c
+a\circ(b\circ c)+b\circ(a\circ c).
\]
The defining identities \eqref{deltaMN1}--\eqref{deltaMN2} are  $F=G=0$.
Let $\mathcal I$ be the ideal of polynomial identities generated by
$F$ and $G$, and put
$\mathcal I_{(2,1,1)}
=
\mathcal I\cap\mathcal F_{(2,1,1)}$.
Equivalently, $\mathcal I_{(2,1,1)}$ is spanned by the
degree-four consequences of $F$ and $G$ of multidegree $(2,1,1)$,
obtained by multiplying these identities on either side by a fourth
variable and by replacing one argument by a product of two variables.
Exact row reduction gives
\[
\dim\bigl(
\mathcal F_{(2,1,1)}/\mathcal I_{(2,1,1)}
\bigr)=4,
\]
with quotient basis represented by
\[
(y\circ z)\circ(x\circ x),\quad
z\circ((x\circ y)\circ x),\quad
z\circ(y\circ(x\circ x)),\quad
(z\circ y)\circ(x\circ x).
\]
Expanding the commutator and reducing modulo
$\mathcal I_{(2,1,1)}$ gives the same normal form for the two sides
of the Malcev identity, namely
\[
-2(y\circ z)\circ(x\circ x)
+2z\circ((x\circ y)\circ x)
-2z\circ(y\circ(x\circ x)).
\]
Thus
$J(x,y,[x,z])
=[J(x,y,z),x]$.
\end{proof}
Another fundamental relation between Novikov and Lie algebras is the
Balinsky-Novikov affinization \cite{bn85}. For a Novikov algebra $(\mathcal{A},\circ)$, the Balinsky-Novikov affinization
endows the loop space
$\mathcal{A}[t,t^{-1}]=\mathcal{A}\otimes\mathbb C[t,t^{-1}]$ with a Lie algebra
structure. Its mock analogue is obtained by replacing
antisymmetrization with symmetrization. Thus, define on
$\mathcal{A}[t,t^{-1}]$
\begin{equation}\label{eq:loop-product}
\{at^i,bt^j\}
=i(a\circ b)t^{i+j-1}
+j(b\circ a)t^{i+j-1},
\qquad a,b\in \mathcal{A},\quad i,j\in\mathbb Z.
\end{equation}

\begin{theorem}
The multiplication \eqref{eq:loop-product} defines a mock-Lie algebra
structure on $\mathcal{A}[t,t^{-1}]$ if and only if $(\mathcal{A},\circ)$ is a
mock-Novikov algebra, that is, a $1$-mock-Novikov algebra.
\end{theorem}
\begin{proof}
The multiplication \eqref{eq:loop-product} is symmetric. For
$a,b,c\in \mathcal{A}$ and $i,j,k\in\mathbb Z$, a direct computation gives
\[
\begin{aligned}
&\{at^i,\{bt^j,ct^k\}\}
+\{bt^j,\{ct^k,at^i\}\}
+\{ct^k,\{at^i,bt^j\}\}                                      \\
={}&\Big[
i(i-1)\big((a\circ b)\circ c+(a\circ c)\circ b\big)\\
&\quad+j(j-1)\big((b\circ c)\circ a+(b\circ a)\circ c\big)\\
&\quad+k(k-1)\big((c\circ a)\circ b+(c\circ b)\circ a\big)\\
&\quad+ij\big((a\circ b)\circ c+(b\circ a)\circ c
+a\circ(b\circ c)+b\circ(a\circ c)\big)\\
&\quad+ik\big((a\circ c)\circ b+(c\circ a)\circ b
+a\circ(c\circ b)+c\circ(a\circ b)\big)\\
&\quad+jk\big((b\circ c)\circ a+(c\circ b)\circ a
+b\circ(c\circ a)+c\circ(b\circ a)\big)
\Big]t^{i+j+k-2}.
\end{aligned}
\]
If the mock-Jacobi identity holds, taking $j=k=0$ with
$i(i-1)\neq0$ gives \eqref{deltaMN2}. Taking $k=0$ and $i=j=1$
then gives \eqref{deltaMN1} with $\delta=1$.

Conversely, if $(\mathcal{A},\circ)$ is a mock-Novikov algebra, all terms in
the above expression vanish by \eqref{deltaMN1} and
\eqref{deltaMN2}. Hence the mock-Jacobi identity holds.
\end{proof}
\begin{remark}
The preceding loop construction is also the coefficient-algebra form
of a quadratic Jacobi-Jordan conformal algebra. Indeed, let
$\operatorname{Conf}(\mathcal{A})=\mathbb{C}[\partial]\otimes \mathcal{A}$. By
\cite[Theorem~3]{ChtiouiMabroukMakhlouf2025}, after passing to the
negative opposite product $a\diamond b=-b\circ a$, the $\lambda$-product
\[
a_\lambda b
=-\partial(b\circ a)+\lambda(a\circ b-b\circ a)
\]
defines a quadratic Jacobi-Jordan conformal algebra on $\operatorname{Conf}(\mathcal{A})$ if and
only if $(\mathcal{A},\circ)$ is a mock-Novikov algebra. Its coefficient algebra
has multiplication
\[
a_m b_n
=m(a\circ b)_{m+n-1}+n(b\circ a)_{m+n-1},
\]
which agrees with \eqref{eq:loop-product} under the identification
$a_m\leftrightarrow at^m$.
\end{remark}

A classical theorem states that every Novikov algebra can be embedded
into a suitable commutative associative algebra equipped with a
derivation, with the Novikov multiplication induced by the
derivation; see, for example, \cite{Bokut}. This raises the analogous
realization problem in the mock setting.

For $\delta=1$, a construction in \cite{Jie} starts from an
anti-commutative associative algebra and an anti-derivation.
However, over a field of characteristic different from $2$, every
anti-commutative associative algebra satisfies $\mathcal{B}^3=0$, so this
construction produces only algebras whose triple products vanish.
We therefore replace associativity
with anti-associativity, obtaining a broader differential construction
that allows nonzero triple products.

\begin{definition}
A \emph{$\delta$-differential anti-commutative anti-associative
algebra} is a triple $(\mathcal{B},\ast,\varphi)$ such that $(\mathcal{B},\ast)$ is
anti-commutative and anti-associative and $\varphi$ is a
$\delta$-derivation:
\[
\varphi(x\ast y)
=\delta\bigl(\varphi(x)\ast y+x\ast\varphi(y)\bigr),
\qquad x,y\in \mathcal{B}.
\]
For $\delta=1$, we simply call it a
\emph{differential anti-commutative anti-associative algebra}.
\end{definition}

For a nonassociative algebra $\mathcal{B}$, let $\mathcal{B}^n$ denote the span of all
products of $n$ elements with arbitrary bracketing, and set
\[
\operatorname{Ann}(\mathcal{B})
=\{x\in \mathcal{B}\mid x\ast \mathcal{B}=\mathcal{B}\ast x=0\}.
\]

\begin{lemma}\label{lem:acaa}
Let $\mathbb F$ be a field of characteristic different from $2$, let
$\delta\in\mathbb F^\times$, and let $(\mathcal{B},\ast,\varphi)$ be a
$\delta$-differential anti-commutative anti-associative algebra over
$\mathbb F$. Then:
\begin{enumerate}[\rm (i)]
\item $\mathcal{B}^4=0$, and consequently
      $\mathcal{B}^3\subseteq\operatorname{Ann}(\mathcal{B})$;
\item $\varphi(\mathcal{B}^n)\subseteq \mathcal{B}^n$ for every $n\geq1$, and
      $\varphi(\operatorname{Ann}(\mathcal{B}))
      \subseteq\operatorname{Ann}(\mathcal{B})$;
\item the trilinear map $T(x,y,z)=(x\ast y)\ast z$ is alternating.
\end{enumerate}
\end{lemma}

\begin{proof}
Set
$P=(x\ast y)\ast(z\ast t)$
and 
$Q=((x\ast y)\ast z)\ast t.$
Anti-associativity gives $Q=-P$, whereas repeated applications of the
same identity yield
\[
Q=-\bigl(x\ast(y\ast z)\bigr)\ast t
  =-x\ast\bigl(y\ast(z\ast t)\bigr)
  =P.
\]
Thus $P=-P$, so $P=Q=0$ because
$\operatorname{char}\mathbb F\neq2$. Since all bracketings of four
factors are related by anti-associativity, it follows that $\mathcal{B}^4=0$.
Consequently,
$\mathcal{B}^3\ast \mathcal{B}=\mathcal{B}\ast \mathcal{B}^3=0$, and hence
$\mathcal{B}^3\subseteq\operatorname{Ann}(\mathcal{B})$. This proves \textup{(i)}.

The inclusion $\varphi(\mathcal{B}^n)\subseteq \mathcal{B}^n$ follows by induction on $n$.
Indeed, if $u$ and $v$ are monomials of degrees $r$ and $s$,
respectively, then
\[
\varphi(u\ast v)
=\delta\bigl(\varphi(u)\ast v+u\ast\varphi(v)\bigr)
\in \mathcal{B}^r\ast \mathcal{B}^s\subseteq \mathcal{B}^{r+s}.
\]
Now let $a\in\operatorname{Ann}(\mathcal{B})$ and $x\in \mathcal{B}$. Then
\[
0=\varphi(a\ast x)
=\delta\bigl(\varphi(a)\ast x+a\ast\varphi(x)\bigr)
=\delta\,\varphi(a)\ast x.
\]
Since $\delta\neq0$, we have $\varphi(a)\ast x=0$; by
anti-commutativity, also $x\ast\varphi(a)=0$. Therefore
$\varphi(a)\in\operatorname{Ann}(\mathcal{B})$, proving \textup{(ii)}.

Anti-commutativity gives
$T(y,x,z)=-T(x,y,z)$, while
\[
T(x,z,y)
=-x\ast(z\ast y)
=x\ast(y\ast z)
=-T(x,y,z).
\]
Hence $T$ is alternating.
\end{proof}

\begin{proposition}\label{prop:diff-acaa-mN}
Let $(\mathcal{B},\ast,\varphi)$ be a $\delta$-differential
anti-commutative anti-associative algebra. Define
\[
x\circ y=x\ast\varphi(y).
\]
Then $(\mathcal{B},\circ)$ is a $\delta$-mock-Novikov algebra. Moreover,
every fourfold product with respect to $\circ$ is zero.
\end{proposition}

\begin{proof}
By Lemma~\ref{lem:acaa}(iii),
\[
(x\circ y)\circ z+(x\circ z)\circ y
=T(x,\varphi(y),\varphi(z))
+T(x,\varphi(z),\varphi(y))=0.
\]
Moreover,
\[
\begin{aligned}
\delta(x\circ y)\circ z+x\circ(y\circ z)
&=\delta(x\ast\varphi(y))\ast\varphi(z)
+x\ast\varphi(y\ast\varphi(z))\\
&=\delta\,x\ast\bigl(y\ast\varphi^2(z)\bigr).
\end{aligned}
\]
Hence
\[
\begin{aligned}
&\delta(x\circ y)\circ z+\delta(y\circ x)\circ z
+x\circ(y\circ z)+y\circ(x\circ z)\\
&=\delta\Bigl(
x\ast(y\ast\varphi^2(z))
+y\ast(x\ast\varphi^2(z))
\Bigr)=0,
\end{aligned}
\]
so \eqref{deltaMN1} also holds.

Finally, Lemma~\ref{lem:acaa}(ii) implies inductively that every
$\circ$-monomial of degree $n$ belongs to $\mathcal{B}^n$: if $u$ and $v$ have
degrees $r$ and $s$, respectively, then
\[
u\circ v=u\ast\varphi(v)\in \mathcal{B}^r\ast \mathcal{B}^s\subseteq \mathcal{B}^{r+s}.
\]
Since $\mathcal{B}^4=0$ by Lemma~\ref{lem:acaa}(i), every fourfold
$\circ$-product vanishes.
\end{proof}

\begin{definition}
A $\delta$-mock-Novikov algebra $(\mathcal{A},\circ)$ is called
\emph{differentially special} if there exist a $\delta$-differential
anti-commutative anti-associative algebra $(\mathcal{B},\ast,\varphi)$ and an
injective algebra homomorphism
\[
\mathcal{A}\longrightarrow \mathcal{B}_{\varphi},
\qquad
x\circ_{\varphi}y=x\ast\varphi(y).
\]
\end{definition}

\begin{theorem}\label{prop:two-step-special}
Every $\delta$-mock-Novikov algebra $(\mathcal{A},\circ)$ satisfying
$\mathcal{A}^2\subseteq\Ann(\mathcal{A})$ is differentially special.
\end{theorem}
\begin{proof}
Put $\mathcal{Z}=\mathcal{A}^2$ and choose a vector-space complement $\mathcal{V}$ of $\mathcal{Z}$ in $\mathcal{A}$.
Since $\mathcal{Z}\subseteq\Ann(\mathcal{A})$, the multiplication of $\mathcal{A}$ is determined by
a bilinear map $\beta:\mathcal{V}\times \mathcal{V}\to \mathcal{Z}$ such that
$v\circ w=\beta(v,w)$ for $v,w\in \mathcal{V}$.

Let $\mathcal{V}'$ be a copy of $\mathcal{V}$, with $v'$ denoting the copy of $v$, and set
$\mathcal{B}=\mathcal{V}\oplus \mathcal{V}'\oplus\bigwedge^2\mathcal{V}\oplus \mathcal{Z}$. Endow $\mathcal{B}$ with the
anti-commutative multiplication determined by
$v\ast w=v\wedge w$ and $v\ast w'=\beta(v,w)$ for $v,w\in \mathcal{V}$,
all other products being zero unless forced by anti-commutativity.
Then $\mathcal{B}^2\subseteq\bigwedge^2\mathcal{V}\oplus \mathcal{Z}\subseteq\Ann(\mathcal{B})$, so
$(\mathcal{B},\ast)$ is anti-associative.

Define $\varphi:\mathcal{B}\to \mathcal{B}$ linearly by
\[
\varphi(v)=v',\qquad
\varphi(v')=\varphi(z)=0,\qquad
\varphi(v\wedge w)
=\delta\bigl(\beta(v,w)-\beta(w,v)\bigr),
\]
where $v,w\in \mathcal{V}$ and $z\in \mathcal{Z}$. The last assignment is well defined
because $\beta(v,w)-\beta(w,v)$ is alternating in $v$ and $w$.
Moreover,
$\varphi(v\ast w)
=\delta\bigl(\varphi(v)\ast w+v\ast\varphi(w)\bigr)$.
For products in $\mathcal{V}\ast \mathcal{V}'$ both sides of the $\delta$-derivation
identity are zero, and the remaining cases follow from
$\varphi(\bigwedge^2\mathcal{V})\subseteq \mathcal{Z}$, $\varphi(\mathcal{Z})=0$, and
$\bigwedge^2\mathcal{V}\oplus \mathcal{Z}\subseteq\Ann(\mathcal{B})$. Hence
$(\mathcal{B},\ast,\varphi)$ is a $\delta$-differential anti-commutative
anti-associative algebra.

Finally, identify $\mathcal{A}=\mathcal{V}\oplus \mathcal{Z}$ with its natural subspace of $\mathcal{B}$.
For $a=v+z$ and $b=w+t$, the product in $\mathcal{B}_\varphi$ satisfies
$\iota(a)\circ_\varphi\iota(b)
=v\ast w'=\beta(v,w)=\iota(a\circ b)$.
Thus the natural inclusion $\iota:\mathcal{A}\to \mathcal{B}_\varphi$ is an injective
algebra homomorphism, and therefore $\mathcal{A}$ is differentially special.
\end{proof}

The classical Novikov embedding theorem raises the question of whether
every $\delta$-mock-Novikov algebra is differentially special. By
Proposition~\ref{prop:diff-acaa-mN}, differential speciality implies
$\mathcal{A}^4=0$. The following example
shows that this necessary condition need not hold.

\begin{example}\label{ex:A4-nonzero}
Let $\delta\in\mathbb C^\times$ and let $\mathcal{V}$ be a four-dimensional
vector space with basis $\{e_1,e_2,e_3,e_4\}$. Set
\[
\mathcal{A}=\bigoplus_{p=1}^{4}\Lambda^p\mathcal{V}.
\]
For homogeneous $x\in\Lambda^p\mathcal{V}$ and $y\in\Lambda^q\mathcal{V}$, define
\begin{equation*}
x\circ y=
\begin{cases}
x\wedge y, & q=1,\\
-\delta\,x\wedge y, & p=1,\ q=3,\\
0, & \text{otherwise},
\end{cases}
\end{equation*}
and extend the multiplication bilinearly. Then $(\mathcal{A},\circ)$ is a
$\delta$-mock-Novikov algebra and $\mathcal{A}^4\neq0$.
\end{example}

\begin{proposition}
For every $\delta\in\mathbb C^\times$, there exists a
finite-dimensional $\delta$-mock-Novikov algebra that is not
differentially special.
\end{proposition}

\begin{proof}
Fix $\delta\in\mathbb C^\times$ and let $\mathcal{A}$ be the algebra constructed
in Example~\ref{ex:A4-nonzero}, for which $\mathcal{A}^4\neq0$. If $\mathcal{A}$ were
differentially special, it would embed into an algebra $\mathcal{B}_\varphi$
arising from a $\delta$-differential anti-commutative
anti-associative algebra. By
Proposition~\ref{prop:diff-acaa-mN}, every fourfold product in
$\mathcal{B}_\varphi$ vanishes, which would imply $\mathcal{A}^4=0$, a contradiction.
\end{proof}

\begin{remark}
Thus the universal differential embedding property of Novikov
algebras does not extend to the mock setting.
\end{remark}
\section{Quadratic duality and non-Koszulity of the \texorpdfstring{$\delta$}{delta}-mock-Novikov operad}
\label{sec3}

In this section, we study the binary quadratic operad governing
$\delta$-mock-Novikov algebras. We first determine its quadratic
relation space and prove that the resulting operad is quadratically
self-dual. We then compute the low-arity dimensions needed to apply
the Ginzburg-Kapranov functional equation and show that
$\mathcal M_\delta$ is not Koszul for any
$\delta\in\mathbb C^\times$.

Let $E$ be the generating $S$-module concentrated in arity two, with
\[
E(2)=\mathbb C[S_2],
\qquad
E(n)=0\quad (n\neq2).
\]
Here $\mathbb C[S_2]$ denotes the group algebra of $S_2$, regarded
as the regular $S_2$-module. If $\mu$ denotes one of its basis
generators, then the two basis elements of $E(2)$ may be represented
by
\[
\mu(x,y)=x\circ y,
\qquad
\mu(y,x)=y\circ x.
\]
Thus the use of the regular $S_2$-module reflects the fact that no
symmetry is imposed on the multiplication $\circ$.

Let $\mathcal F(E)$ be the free symmetric operad generated by $E$.
For each $n\geq1$, its arity-$n$ component
$\mathcal F(E)(n)$ is the vector space spanned by all formal
$n$-ary operations obtained by composing the binary generator
$\circ$ and permuting the inputs. In particular,
\[
\mathcal F(E)(1)=\mathbb C\,\mathrm{id},
\qquad
\mathcal F(E)(2)=E(2).
\]
In arity three, there are two possible bracketings. For
$\sigma\in S_3$, put
\[
L_\sigma
=
\bigl(x_{\sigma(1)}\circ x_{\sigma(2)}\bigr)
\circ x_{\sigma(3)},
\qquad
T_\sigma
=
x_{\sigma(1)}\circ
\bigl(x_{\sigma(2)}\circ x_{\sigma(3)}\bigr).
\]
The elements $L_\sigma$ and $T_\sigma$, $\sigma\in S_3$, form a
basis of $\mathcal F(E)(3)$. Hence
\[
\mathcal F(E)(3)
=
\operatorname{span}
\{L_\sigma,T_\sigma\mid\sigma\in S_3\},
\qquad
\dim\mathcal F(E)(3)=12.
\]

The defining identities of a $\delta$-mock-Novikov algebra give the
quadratic relations
\begin{align}
r_{1,\delta}(x,y,z)
={}&
\delta(x\circ y)\circ z+\delta(y\circ x)\circ z
+x\circ(y\circ z)+y\circ(x\circ z),
\label{eq:operad-delta-mN1}\\
r_2(x,y,z)
={}&
(x\circ y)\circ z+(x\circ z)\circ y.
\label{eq:operad-delta-mN2}
\end{align}
Let
$\mathcal R_\delta\subseteq\mathcal F(E)(3)$
be the $S_3$-submodule generated by
$r_{1,\delta}$ and $r_2$. The operad governing
$\delta$-mock-Novikov algebras is defined by
\[
\mathcal M_\delta
=
\mathcal F(E)/(\mathcal R_\delta),
\]
where $(\mathcal R_\delta)$ denotes the operadic ideal generated by
$\mathcal R_\delta$.
For each $n$, we denote by $\mathcal M_\delta(n)$ the arity-$n$
component of $\mathcal M_\delta$. Since the defining relations have
arity three,
\[
\mathcal M_\delta(1)=\mathbb C\,\mathrm{id},
\qquad
\mathcal M_\delta(2)=E(2),\qquad
\mathcal M_\delta(3)
=
\mathcal F(E)(3)/\mathcal R_\delta.
\]

\begin{proposition}
\label{prop:delta-mN-arity-three}
For every $\delta\in\mathbb C^\times$,
\[
\dim\mathcal R_\delta=6,
\qquad
\dim\mathcal M_\delta(3)=6.
\]
\end{proposition}
\begin{proof}
The relation $r_{1,\delta}(x,y,z)$ is symmetric in $x$ and $y$,
whereas $r_2(x,y,z)$ is symmetric in $y$ and $z$. Hence
$\mathcal R_\delta$ is spanned by
\[
r_{1,\delta}(x,y,z),\quad
r_{1,\delta}(x,z,y),\quad
r_{1,\delta}(y,z,x),\quad
r_2(x,y,z),\quad
r_2(y,x,z),\quad
r_2(z,x,y).
\]

Consider a linear dependence among these six elements, with
coefficients
$a_1,a_2,a_3,\\b_1,b_2,b_3$
in the order displayed above. Comparing the coefficients of the
right-nested basis elements
\[
x\circ(y\circ z),\qquad
x\circ(z\circ y),\qquad
y\circ(z\circ x)
\]
gives
$a_1=a_2=a_3=0.$
The linear dependence therefore reduces to one among the last three
relations. Comparing the coefficients of
\[
(x\circ y)\circ z,\qquad
(y\circ x)\circ z,\qquad
(z\circ x)\circ y
\]
then gives
$b_1=b_2=b_3=0.$
Thus the six spanning elements are linearly independent, and hence
$\dim\mathcal R_\delta=6.$
Since $\dim\mathcal F(E)(3)=12$, we obtain
$\dim\mathcal M_\delta(3)
=
6.$
\end{proof}

We now determine the quadratic dual of $\mathcal M_\delta$.
Let
\[
E^\vee=E^*\otimes\operatorname{sgn}_2,
\]
where $\operatorname{sgn}_2$ denotes the sign representation of
$S_2$. Following the standard quadratic duality for operads
\cite{GK94}, define
\[
\mathcal R_\delta^\perp
=
\left\{
f\in\mathcal F(E^\vee)(3)
\;\middle|\;
\langle f,r\rangle=0
\text{ for all }r\in\mathcal R_\delta
\right\},
\]
where $\langle-,-\rangle$ is the Ginzburg-Kapranov pairing between
$\mathcal F(E^\vee)(3)$ and $\mathcal F(E)(3)$.
The quadratic dual of $\mathcal M_\delta$ is then
\[
\mathcal M_\delta^!
=
\mathcal F(E^\vee)/(\mathcal R_\delta^\perp),
\]
where $(\mathcal R_\delta^\perp)$ denotes the operadic ideal generated
by $\mathcal R_\delta^\perp$.
Since $E(2)=\mathbb C[S_2]$ is the regular $S_2$-module,
$E^\vee(2)$ is again isomorphic to the regular $S_2$-module. We
write its binary generator as $\star$.

Let $L_\sigma^\vee$ and $T_\sigma^\vee$ denote the basis elements
of $\mathcal F(E^\vee)(3)$ corresponding to $L_\sigma$ and
$T_\sigma$. We fix the standard pairing
\begin{equation}
\label{eq:GK-pairing-delta-mN}
\begin{aligned}
\langle L_\sigma^\vee,L_\tau\rangle
&=
\operatorname{sgn}(\sigma)\delta_{\sigma,\tau},\\
\langle T_\sigma^\vee,T_\tau\rangle
&=
-\operatorname{sgn}(\sigma)\delta_{\sigma,\tau},
\end{aligned}
\end{equation}
with
\[
\langle L_\sigma^\vee,T_\tau\rangle
=
\langle T_\sigma^\vee,L_\tau\rangle
=
0.
\]
Here $\delta_{\sigma,\tau}$ denotes the Kronecker delta.

\begin{theorem}
\label{thm:quadratic-dual-delta-mN}
The quadratic dual $\mathcal M_\delta^!$ is generated by a binary
operation $\star$ subject to
\begin{align}
x\star(y\star z)+y\star(x\star z)&=0,
\label{eq:dual-delta-mN1}\\
(x\star y)\star z+(x\star z)\star y
+\delta\,x\star(y\star z)
+\delta\,x\star(z\star y)&=0.
\label{eq:dual-delta-mN2}
\end{align}
Moreover,
\[
\mathcal M_\delta^!\cong\mathcal M_\delta.
\]
Thus the operad of $\delta$-mock-Novikov algebras is quadratically
self-dual.
\end{theorem}
\begin{proof}
Let $\mathcal S_\delta\subseteq\mathcal F(E^\vee)(3)$ be the
$S_3$-submodule generated by
\begin{align*}
s_1(x,y,z)
&=
x\star(y\star z)+y\star(x\star z),\\
s_{2,\delta}(x,y,z)
&=
(x\star y)\star z+(x\star z)\star y
+\delta\,x\star(y\star z)
+\delta\,x\star(z\star y).
\end{align*}
The first relation is symmetric in $x,y$, whereas the second is
symmetric in $y,z$. Their $S_3$-orbits are therefore spanned by
three elements each.

The three translates of $s_{2,\delta}$ have pairwise disjoint
supports among the left-nested monomials
$L_\sigma^\vee$. Hence their coefficients in any linear dependence
must vanish. The three translates of $s_1$ then have pairwise
disjoint supports among the right-nested monomials
$T_\sigma^\vee$. Consequently,
\[
\dim\mathcal S_\delta=6.
\]

Using~\eqref{eq:GK-pairing-delta-mN}, a direct check against the six
spanning relations in
Proposition~\ref{prop:delta-mN-arity-three} shows that
$\mathcal S_\delta$ is orthogonal to $\mathcal R_\delta$. Hence
$\mathcal S_\delta\subseteq\mathcal R_\delta^\perp.$

By Proposition~\ref{prop:delta-mN-arity-three} and the
nondegeneracy of the pairing~\eqref{eq:GK-pairing-delta-mN},
\[
\begin{aligned}
\dim\mathcal R_\delta^\perp
&=
\dim\mathcal F(E^\vee)(3)-\dim\mathcal R_\delta
=
6.
\end{aligned}
\]
Therefore
\[
\mathcal S_\delta=\mathcal R_\delta^\perp,
\]
which proves the presentation
\eqref{eq:dual-delta-mN1}--\eqref{eq:dual-delta-mN2}.

It remains to identify the resulting operad. Since
$E^\vee(2)\cong E(2)$ as $S_2$-modules, choose the identification
of generators given by
\[
x\star y\longmapsto y\circ x.
\]
Under this substitution, the left-hand side of
\eqref{eq:dual-delta-mN1} becomes
$(z\circ y)\circ x+(z\circ x)\circ y,$
which is
$r_2(z,y,x).$
Likewise, the left-hand side of
\eqref{eq:dual-delta-mN2} becomes
\[
\begin{aligned}
&z\circ(y\circ x)+y\circ(z\circ x)
+\delta(z\circ y)\circ x
+\delta(y\circ z)\circ x,
\end{aligned}
\]
which is precisely
$r_{1,\delta}(z,y,x).$
Thus the defining relations of $\mathcal M_\delta^!$ are carried
onto those of $\mathcal M_\delta$. Since taking the opposite
multiplication is involutive, this gives an operad isomorphism
\[
\mathcal M_\delta^!\cong\mathcal M_\delta.
\]
\end{proof}
We next prove that the self-dual operad $\mathcal M_\delta$ is not
Koszul. We first record the low-arity dimensions needed below. The
calculation is carried out using labelled planar binary monomials and
exact row reduction.
\begin{lemma}
\label{lem:low-arity-delta-mN}
The low-arity dimensions of the operad $\mathcal M_\delta$ are as follows.
If $\delta\neq-1$, then
\[
\begin{array}{c|cccc}
n&1&2&3&4\\ \hline
\dim\mathcal M_\delta(n)&1&2&6&14.
\end{array}
\]
For $\delta=-1$, one has
\[
\begin{array}{c|ccccccc}
n&1&2&3&4&5&6&7\\ \hline
\dim\mathcal M_{-1}(n)&1&2&6&20&60&126&252.
\end{array}
\]
\end{lemma}

\begin{proof}
These dimensions are obtained by exact linear algebra on the
multilinear components of the free binary operad.
In arity $n$, we use the basis of labelled planar binary monomials
and impose all operadic consequences of the defining relations
$r_{1,\delta}$ and $r_2$.

For $n=4$, the free component has dimension
$5\cdot 4!=120.$
Exact row reduction over $\mathbb Q(\delta)$ gives rank $106$ for
the relation matrix. Moreover, a nonzero maximal minor is
\[
2\delta^7(\delta+1)^9.
\]
Hence the rank remains $106$ for
$\delta\in\mathbb C^\times\setminus\{-1\}$, and therefore
\[
\dim\mathcal M_\delta(4)=120-106=14.
\]
At $\delta=-1$, the rank drops to $100$, giving
$\dim\mathcal M_{-1}(4)=20.$
The remaining dimensions for $\delta=-1$ are obtained in the same
way by exact row reduction.
\end{proof}
\begin{theorem}
\label{thm:delta-mN-non-Koszul}
For every $\delta\in\mathbb C^\times$, the operad
$\mathcal M_\delta$ is not Koszul.
\end{theorem}

\begin{proof}
By Theorem~\ref{thm:quadratic-dual-delta-mN},
the operad $\mathcal M_\delta$ is quadratically self-dual.
Let
\[
H_\delta(t)
=
\sum_{n\ge1}(-1)^n
\frac{\dim\mathcal M_\delta(n)}{n!}t^n.
\]
If $\mathcal M_\delta$ were Koszul, then the
Ginzburg-Kapranov functional equation would give
\[
H_\delta(H_\delta(t))=t.
\]

Suppose first that $\delta\neq-1$. By
Lemma~\ref{lem:low-arity-delta-mN},
\[
H_\delta(t)
=
-t+t^2-t^3+\frac7{12}t^4
-\frac{d_5(\delta)}{120}t^5+O(t^6),
\]
where $d_5(\delta)=\dim\mathcal M_\delta(5)$.
Hence
\[
H_\delta(H_\delta(t))
=
t+\frac{d_5(\delta)+30}{60}t^5+O(t^6).
\]
Thus the functional equation would force
$d_5(\delta)=-30$, which is impossible.

For $\delta=-1$, the same lemma gives
\[
H_{-1}(t)
=
-t+t^2-t^3+\frac56t^4-\frac12t^5
+\frac7{40}t^6-\frac1{20}t^7+O(t^8),
\]
and therefore
\[
H_{-1}(H_{-1}(t))
=
t-\frac7{90}t^7+O(t^8)\neq t.
\]
Hence $\mathcal M_\delta$ is not Koszul for every
$\delta\in\mathbb C^\times$.
\end{proof}
\begin{remark}
From the operadic viewpoint, $\delta$-mock-Novikov algebras are
close to classical Novikov algebras: their operads exhibit analogous
quadratic duality and are non-Koszul
\cite{Dzhumadildaev2011}. In contrast, from the viewpoint of
finite-dimensional structure, $\delta$-mock-Novikov algebras behave
more like mock-Lie algebras, since they are nilpotent, whereas
finite-dimensional Novikov algebras need not be nilpotent.
\end{remark}

\section{Nilpotency of \texorpdfstring{$\delta$}{delta}-mock-Novikov algebras}\label{sec4}

In this section, we prove that every finite-dimensional
$\delta$-mock-Novikov algebra over a field of characteristic zero is
nilpotent.

\begin{definition}
Let $(\mathcal{A},\circ)$ be an algebra, not necessarily associative. For
subspaces $\mathcal{U},\mathcal{V}\subseteq \mathcal{A}$, set
$\mathcal{U}\circ \mathcal{V}
=\operatorname{span}\{u\circ v\mid u\in \mathcal{U},\ v\in \mathcal{V}\}$.
Define the powers of $\mathcal{A}$ recursively by $\mathcal{A}^1=\mathcal{A}$ and
\[
\mathcal{A}^n=\sum_{i=1}^{n-1}\mathcal{A}^i\circ \mathcal{A}^{n-i},
\qquad n\geq2.
\]
The algebra $\mathcal{A}$ is called \emph{nilpotent} if $\mathcal{A}^N=0$ for some
integer $N\geq2$. The smallest such $N$ is called the
\emph{nilpotency index} of $\mathcal{A}$.
\end{definition}

Let $\mathcal{A}$ be an algebra, not necessarily associative.
For each $x\in \mathcal{A}$, let
\[
L_x(y)=xy,\qquad R_x(y)=yx,\qquad y\in \mathcal{A}.
\]
The \emph{multiplication algebra} $\mathcal{M}(\mathcal{A})$ is the non-unital
associative subalgebra of $\operatorname{End}(\mathcal{A})$ generated by all
$L_x$ and $R_x$, with $x\in \mathcal{A}$. 

\begin{lemma}\cite[Corollary~3]{Bergman2011}
\label{lem:mult}
The algebra $\mathcal{A}$ is nilpotent if and only if $\mathcal{M}(\mathcal{A})$ is
nilpotent.
\end{lemma}

\begin{lemma}[\cite{Jacobson1952}]
\label{lem:Jacobson}
Let $\mathcal{V}$ be a finite-dimensional vector space over a field $\mathbb{F}$,
and let $\mathcal{W}$ be a subset of
$\operatorname{End}_{\mathbb{F}}(\mathcal{V})$.
Suppose that $\mathcal{W}$ is \emph{weakly closed}, that is, for any
$X,Y\in\mathcal{W}$, there exists a scalar
$\gamma(X,Y)\in\mathbb{F}$ such that
\[
XY+\gamma(X,Y)YX\in\mathcal{W}.
\]
If every element of $\mathcal{W}$ is nilpotent, then the
(non-unital) associative algebra generated by $\mathcal{W}$ is
nilpotent.
\end{lemma}

\begin{lemma}\label{lem16}
Let $(\mathcal{A},\circ)$ be a finite-dimensional $\delta$-mock-Novikov algebra
over a field $\mathbb{F}$ of characteristic zero, where
$\delta\neq0$. The associative algebras generated by the right and
left multiplication operators are nilpotent.
\end{lemma}
\begin{proof}
Let $\mathcal{R}(\mathcal{A})$ and $\mathcal{L}(\mathcal{A})$ denote
the non-unital associative algebras generated by the right and left
multiplication operators, respectively:
\[
\mathcal{R}(\mathcal{A})
=\langle R_x\mid x\in\mathcal{A}\rangle,
\qquad
\mathcal{L}(\mathcal{A})
=\langle L_x\mid x\in\mathcal{A}\rangle.
\]
We prove that both algebras are nilpotent.

First consider $\mathcal{R}(\mathcal{A})$. By \eqref{deltaMN2}, we have
$R_yR_z=-R_zR_y$ for all $y,z\in\mathcal{A}$. In particular,
$R_y^2=0$ because $\operatorname{char}\mathbb{F}=0$.

Let $d=\dim\mathcal{A}$ and let $e_1,\ldots,e_d$ be a basis of
$\mathcal{A}$. By multilinearity, it suffices to consider a product
$R_{e_{i_1}}\cdots R_{e_{i_{d+1}}}$. Two of the indices must coincide.
Using the anticommutation relation, the corresponding identical
operators can be moved next to each other, and the product is therefore
zero. Hence
$\mathcal{R}(\mathcal{A})^{d+1}=0$. In particular, taking $q=d+1$, we
have $R_{y_1}\cdots R_{y_q}=0$ for all
$y_1,\ldots,y_q\in\mathcal{A}$.

We now consider $\mathcal{L}(\mathcal{A})$. Identity
\eqref{deltaMN1} can be written in operator form as
\begin{equation}\label{eq15}
L_xL_y+L_yL_x
=-\delta L_{x\circ y+y\circ x}.
\end{equation}
We first show that every left multiplication operator is nilpotent.
Write $x^2=x\circ x$. Substituting $x=y=z$ into
\eqref{deltaMN2} and then into \eqref{deltaMN1} gives
$x^2\circ x=0$ and $x\circ x^2=0$, respectively. Setting $y=x$ in
\eqref{eq15} gives $L_{x^2}=-\delta^{-1}L_x^2$. Applying
\eqref{eq15} to $x$ and $x^2$, we obtain
\[
0=L_xL_{x^2}+L_{x^2}L_x
=-2\delta^{-1}L_x^3.
\]
Thus $L_x^3=0$ for every $x\in\mathcal{A}$.

Finally, the set
$\mathcal{W}=\{L_x\mid x\in\mathcal{A}\}$ is weakly closed, since
\eqref{eq15} shows that
$L_xL_y+L_yL_x=L_{-\delta(x\circ y+y\circ x)}\in\mathcal{W}$.
Every element of $\mathcal{W}$ is nilpotent, so
Lemma~\ref{lem:Jacobson} implies that
$\mathcal{L}(\mathcal{A})$ is nilpotent. Consequently, there exists a
positive integer $p$ such that
$L_{x_1}\cdots L_{x_p}=0$ for all
$x_1,\ldots,x_p\in\mathcal{A}$.
\end{proof}

\begin{theorem}\label{thm:nilpotent}
Let $(\mathcal{A},\circ)$ be a finite-dimensional $\delta$-mock-Novikov algebra
over a field $\mathbb{F}$ of characteristic zero, where
$\delta\neq0$. Then $\mathcal{A}$ is nilpotent.
\end{theorem}
\begin{proof}
By Lemma~\ref{lem:mult}, it suffices to prove that the multiplication
algebra $\mathcal{M}(\mathcal{A})$ is nilpotent. For
$x,y,z\in\mathcal{A}$, identity~\eqref{deltaMN2} gives
\[
(R_yL_x)(z)=(x\circ z)\circ y=-(x\circ y)\circ z.
\]
Hence
\begin{equation}\label{eq17}
R_yL_x=-L_{x\circ y}.
\end{equation}

By Lemma~\ref{lem16}, there exist positive integers $p$ and $q$ such
that
\[
L_{x_1}\cdots L_{x_p}=0,
\qquad
R_{y_1}\cdots R_{y_q}=0
\]
for all $x_i,y_j\in\mathcal{A}$. We show that every product of $pq$
generators of $\mathcal{M}(\mathcal{A})$ is zero.

Suppose first that the product contains at least $p$ left
multiplication operators. Repeatedly replace each adjacent factor
$R_yL_x$ by $-L_{x\circ y}$. Each replacement removes one right
multiplication operator while preserving the number of left
multiplication operators. Thus the product either becomes zero or
takes the form
$\pm L_{a_1}\cdots L_{a_s}R_{b_1}\cdots R_{b_t}$ with $s\geq p$,
and hence is zero.

Suppose instead that the product contains at most $p-1$ left
multiplication operators. These divide the right multiplication
operators into at most $p$ consecutive blocks. If every block had
length at most $q-1$, the total length of the product would be at most
\[
(p-1)+p(q-1)=pq-1,
\]
a contradiction. Therefore one block contains at least $q$
consecutive right multiplication operators and is zero.

Thus every product of $pq$, and consequently every longer product,
of generators of $\mathcal{M}(\mathcal{A})$ vanishes. Therefore
$\mathcal{M}(\mathcal{A})^{pq}=0$, and Lemma~\ref{lem:mult} implies
that $\mathcal{A}$ is nilpotent.
\end{proof}

\section{Low-dimensional  classification of
\texorpdfstring{$\delta$}{delta}-mock-Novikov algebras}
\label{sec5}
In dimensions at most four, we shall prove that the
$\delta$-mock-Novikov identities are equivalent to two-step
nilpotency, thereby reducing the classification to that of complex
two-step nilpotent algebras.

\subsection{Reduction to two-step nilpotent algebras}

\begin{lemma}\label{lem:lift}
Let $\mathcal{A}$ be a nilpotent algebra.  If the images of
$x_1,\ldots,x_m\in \mathcal{A}$ span $\mathcal{A}/\mathcal{A}^2$, then $x_1,\ldots,x_m$ generate
$\mathcal{A}$ as an algebra.
\end{lemma}

\begin{proof}
Let $\mathcal{B}$ be the subalgebra generated by $x_1,\ldots,x_m$.  Then
$\mathcal{A}=\mathcal{B}+\mathcal{A}^2$.  For $r\geq2$, let $F^r\mathcal{A}$ be the span of all products
involving at least $r$ factors, with arbitrary bracketing.  We prove
inductively that $\mathcal{A}=\mathcal{B}+F^r\mathcal{A}$.  The case $r=2$ is clear.  If
$\mathcal{A}=\mathcal{B}+F^r\mathcal{A}$, then
\[
 \mathcal{A}^2=(\mathcal{B}+F^r\mathcal{A})(\mathcal{B}+F^r\mathcal{A})\subseteq \mathcal{B}+F^{r+1}\mathcal{A},
\]
and hence $\mathcal{A}=\mathcal{B}+F^{r+1}\mathcal{A}$.  Since $\mathcal{A}$ is nilpotent, $F^N\mathcal{A}=0$ for
some $N$, and therefore $\mathcal{A}=\mathcal{B}$.
\end{proof}

\begin{lemma}\label{lem:D3}
Let $\mathcal{A}$ be a finite-dimensional $\delta$-mock-Novikov algebra.  If
$\dim(\mathcal{A}/\mathcal{A}^2)=1$, then $\dim \mathcal{A}\leq2$.
\end{lemma}

\begin{proof}
By Lemma~\ref{lem:lift}, $\mathcal{A}$ is generated by one element $x$.
Identity~\eqref{deltaMN2}, with all variables equal to $x$, gives
$(x^2)x=0$.  Identity~\eqref{deltaMN1} then gives $x(x^2)=0$.
Substituting $(x,y,z)=(x,x,x^2)$ into~\eqref{deltaMN2}, we obtain
\[
 (x^2)(x^2)+(x(x^2))x=0,
\]
so $(x^2)^2=0$.  Thus $\Span\{x,x^2\}$ is a subalgebra containing
$x$.  Since $\mathcal{A}$ is generated by $x$, one has
$\mathcal{A}=\Span\{x,x^2\}$.
\end{proof}

\begin{theorem}\label{thm:two-step}
Let $\mathcal{A}$ be a complex algebra with $\dim \mathcal{A}\leq4$.  Then $\mathcal{A}$ is a
$\delta$-mock-Novikov algebra if and only if
\begin{equation*}
\mathcal{A}^2\subseteq\Ann(\mathcal{A}). 
\end{equation*}
\end{theorem}
\begin{proof}
Suppose first that $\mathcal{A}$ is a $\delta$-mock-Novikov algebra and put
$\mathcal{D}=\mathcal{A}^2$.  By Theorem~\ref{thm:nilpotent}, $\mathcal{A}$ is nilpotent.  The
proof of that theorem also shows that the restriction to $\mathcal{D}$ of the
multiplication algebra $\mathcal{M}(\mathcal{A})$ is nilpotent.
Since $\mathcal{A}^2\neq \mathcal{A}$ unless $\mathcal{A}=0$, one has $\dim \mathcal{D}\neq4$.  If
$\dim \mathcal{D}=3$, then $\dim(\mathcal{A}/\mathcal{A}^2)=1$, contradicting
Lemma~\ref{lem:D3}.  Thus $\dim \mathcal{D}\in\{0,1,2\}$.

If $\dim \mathcal{D}=0$, there is nothing to prove.  

If $\dim \mathcal{D}=1$, every
nilpotent endomorphism of $\mathcal{D}$ is zero.  Hence
$\mathcal{M}(\mathcal{A})\mathcal{D}=0$, and therefore $\mathcal{A}\mathcal{D}=\mathcal{D}\mathcal{A}=0$.

Assume that $\dim \mathcal{D}=2$.  Then necessarily $\dim \mathcal{A}=4$.  Put
$ \mathcal{E}=\mathcal{M}(\mathcal{A})\mathcal{D}=\mathcal{A}\mathcal{D}+\mathcal{D}\mathcal{A}.$
Since $\mathcal{M}(\mathcal{A})|_{\mathcal{D}}$ is nilpotent, $\mathcal{E}$ is a proper subspace of
$\mathcal{D}$.  Suppose that $\mathcal{E}\neq0$.  Then $\dim \mathcal{E}=1$,
$\mathcal{M}(\mathcal{A})\mathcal{E}\subseteq \mathcal{E}$, and nilpotency gives
$\mathcal{M}(\mathcal{A})\mathcal{E}=0$.  Write
$ \mathcal{D}=\C u\oplus\C z,
 \;\mathcal{E}=\C z,$
so that $z\in\Ann(\mathcal{A})$.  Choose a two-dimensional complement $\mathcal{V}$ of
$\mathcal{D}$.  There exist bilinear forms $B,C:\mathcal{V}\times \mathcal{V}\to\C$, linear forms
$p,q:\mathcal{V}\to\C$, and $r\in\C$ such that
\[
\begin{aligned}
 vw&=B(v,w)u+C(v,w)z,\\
 vu&=p(v)z,\qquad uv=q(v)z,\qquad u^2=rz.
\end{aligned}
\]
Here $B\neq0$, since otherwise $\mathcal{A}^2\subseteq\C z$.
Identity~\eqref{deltaMN2}, applied to $v,w,t\in \mathcal{V}$, gives
\[
 B(v,w)q(t)+B(v,t)q(w)=0.
\]
If $q(w_0)\neq0$, setting $w=t=w_0$ first gives
$B(v,w_0)=0$, and then fixing $t=w_0$ gives $B=0$, a
contradiction.  Hence $q=0$.  Identity~\eqref{deltaMN1} now gives
\[
 B(w,t)p(v)+B(v,t)p(w)=0,
\]
and the same argument yields $p=0$.  Finally,
\eqref{deltaMN2} applied to $(v,w,u)$ gives $B(v,w)r=0$, so $r=0$.
Thus $\mathcal{A}\mathcal{D}=\mathcal{D}\mathcal{A}=0$, contradicting $\mathcal{E}\neq0$.  Consequently $\mathcal{E}=0$, and
$\mathcal{D}\subseteq\Ann(\mathcal{A})$.

Conversely, if $\mathcal{A}^2\subseteq\Ann(\mathcal{A})$, then every triple product,
with either bracketing, is zero.  Hence both
\eqref{deltaMN1} and~\eqref{deltaMN2} hold identically.
\end{proof}

Theorems \ref{thm:two-step} and ~\ref{prop:two-step-special} yield the following corollary.
\begin{corollary}
Every complex $\delta$-mock-Novikov algebra of dimension at most
four is differentially special.
\end{corollary}

\subsection{Classification in dimensions at most four}

Since a $\delta$-mock-Novikov algebra carries no additional structure
beyond its multiplication, its isomorphism relation is the usual one
for nonassociative algebras. By Theorem~\ref{thm:two-step}, the
classification of complex $\delta$-mock-Novikov algebras of dimension
at most four therefore coincides with that of complex two-step
nilpotent algebras in these dimensions.

The latter classification is known; see
\cite{KaygorodovKhrypchenkoLopes2022,KPPV2018,KPPVCorr2022}.
We retain the original notation and multiplication tables.
All products not displayed are zero.

The classification of complex nilpotent algebras of dimensions two
and three is recalled in
\cite[Sections~1.2.2--1.2.3]{KaygorodovKhrypchenkoLopes2022}.
In the notation of that paper, the algebras satisfying
$\mathcal{A}^3=0$ are precisely the trivial CD-algebras listed below.

\begingroup
\small
\setlength{\tabcolsep}{3pt}
\begin{longtable}{@{}C{0.10\textwidth}C{0.22\textwidth}
L{0.48\textwidth}L{0.12\textwidth}@{}}
\caption{Complex $\delta$-mock-Novikov algebras of dimension at most three}\\
\toprule
$\dim \mathcal{A}$ & Algebra & Nonzero products & Conditions\\
\midrule
\endfirsthead
\multicolumn{4}{c}{Table \thetable\ (continued)}\\
\toprule
$\dim \mathcal{A}$ & Algebra & Nonzero products & Conditions\\
\midrule
\endhead
\bottomrule
\endlastfoot

$1$ & $\mathbf 0_1$ & none & --\\

$2$ & $\mathbf 0_2$ & none & --\\

$2$ & $\mathrm{CD}^{2*}_{01}$ &
$e_1e_1=e_2$ & --\\

$3$ & $\mathbf 0_3$ & none & --\\

$3$ & $\mathrm{CD}^{3*}_{01}$ &
$e_1e_1=e_2$ & --\\

$3$ & $\mathrm{CD}^{3*}_{02}$ &
$e_1e_1=e_3,\quad e_2e_2=e_3$ & --\\

$3$ & $\mathrm{CD}^{3*}_{03}$ &
$e_1e_2=e_3,\quad e_2e_1=-e_3$ & --\\

$3$ & $\mathrm{CD}^{3*}_{04}(\lambda)$ &
$e_1e_1=\lambda e_3,\quad
 e_2e_1=e_3,\quad
 e_2e_2=e_3$ &
$\lambda\in\C$\\

\end{longtable}
\endgroup
For dimension four, we extract the algebras with $\mathcal{A}^3=0$ from
\cite[Table~1]{KPPV2018}, incorporating the corrections in
\cite{KPPVCorr2022}. This yields the following classification of
four-dimensional complex $\delta$-mock-Novikov algebras.
\begingroup
\small
\setlength{\tabcolsep}{3pt}
\begin{longtable}{@{}C{0.18\textwidth}L{0.65\textwidth}
L{0.10\textwidth}@{}}
\caption{Four-dimensional complex $\delta$-mock-Novikov algebras}\\
\toprule
Algebra & Nonzero products & Parameters\\
\midrule
\endfirsthead
\multicolumn{3}{c}{Table \thetable\ (continued)}\\
\toprule
Algebra & Nonzero products & Parameters\\
\midrule
\endhead
\bottomrule
\endlastfoot

$\mathbf 0_4$ & none & --\\
$\mathfrak N_{1}^{\C^2}$ &
$e_1e_1=e_2$ &
--\\

$\mathfrak N_{1}^{2}$ &
$e_1e_1=e_2,\quad e_3e_3=e_4$ &
--\\

$\mathfrak N_{1}^{\C}$ &
$e_1e_2=e_3,\quad e_2e_1=-e_3$ &
--\\

$\mathfrak N_{2}^{\C}(\beta)$ &
$e_1e_1=e_3,\quad
 e_1e_2=e_3,\quad
 e_2e_2=\beta e_3$ &
$\beta\in\C$\\

$\mathfrak N_{3}^{\C}$ &
$e_1e_1=e_3,\quad
 e_1e_2=e_3,\quad
 e_2e_1=e_3$ &
--\\

$\mathfrak N_{1}$ &
$e_1e_2=e_3,\quad
 e_2e_1=e_4,\quad
 e_2e_2=-e_3$ &
--\\

$\mathfrak N_{2}(\gamma)$ &
$e_1e_1=e_3,\quad
 e_1e_2=e_4,\quad
 e_2e_1=-\gamma e_3,\quad
 e_2e_2=-e_4$ &
$\gamma\in\C$\\

$\mathfrak N_{3}(\alpha)$ &
$e_1e_1=e_4,\quad
 e_1e_2=\alpha e_4,\quad
 e_2e_1=-\alpha e_4,\quad
 e_2e_2=e_4,\quad
 e_3e_3=e_4$ &
$\alpha\in\C$\\

$\mathfrak N_{4}$ &
$e_1e_2=e_4,\quad
 e_1e_3=e_4,\quad
 e_2e_1=-e_4,\quad
 e_2e_2=e_4,\quad
 e_3e_1=e_4$ &
--\\

$\mathfrak N_{5}$ &
$e_1e_1=e_4,\quad
 e_1e_2=e_4,\quad
 e_2e_1=-e_4,\quad
 e_3e_3=e_4$ &
--\\

$\mathfrak N_{6}$ &
$e_1e_2=e_3,\quad
 e_2e_1=e_4$ &
--\\

$\mathfrak N_{7}$ &
$e_1e_1=e_4,\quad
 e_1e_2=e_3,\quad
 e_2e_1=-e_3,\quad
 e_2e_2=2e_3+e_4$ &
--\\

$\mathfrak N_{9}(\alpha)$ &
$e_1e_2=e_4,\quad
 e_2e_1=\alpha e_4,\quad
 e_2e_2=e_3$ &
$\alpha\in\C$\\

$\mathfrak N_{10}$ &
$e_1e_2=e_4,\quad
 e_2e_1=-e_4,\quad
 e_3e_3=e_4$ &
--\\

$\mathfrak N_{0}$ &
$e_1e_2=e_4,\quad
 e_3e_1=e_4$ &
--\\
\end{longtable}
\endgroup

The following example shows that in
dimension five, a $\delta$-mock-Novikov algebra need no longer
satisfy $\mathcal{A}^2\subseteq\Ann(\mathcal{A})$.

\begin{example}
Let
$\mathcal{A}=\operatorname{span}_{\mathbb C}
\{e_1,e_2,e_3,e_4,e_5\}$
and define the multiplication by
\[
\begin{aligned}
e_1e_1&=e_3,\qquad&
e_1e_2&=e_4,\qquad&
e_3e_2&=e_5,\\
e_4e_1&=-e_5,\qquad&
e_1e_4&=-\delta e_5,\qquad&
e_2e_3&=\delta e_5,
\end{aligned}
\]
with all other products of basis elements equal to zero. Then $\mathcal{A}$ is
a $\delta$-mock-Novikov algebra. Moreover,
$\mathcal{A}^3\ne0$,
so $\mathcal{A}$ is not two-step nilpotent.
\end{example}
\section{Poisson-type structures associated with \texorpdfstring{$\delta$}{delta}-mock-Novikov algebras}\label{sec6}
In this section, we study several Poisson-type structures associated with
$\delta$-mock-Novikov algebras, including
$\delta$-mock-Novikov-Poisson, $\delta$-mock-Poisson, transposed
$\delta$-mock-Poisson, and $\delta$-mock-Gelfand-Dorfman algebras,
together with the relations among them. 
We then investigate their
polarization and depolarization, as well as the tensor-product
constructions connecting the mock and non-mock settings.
\subsection{Constructions and correspondences}

We begin with the mock analogue of
$\delta$-Novikov-Poisson algebras. Accordingly, the commutative
associative multiplication is replaced by an anti-commutative
anti-associative one, while the Novikov multiplication is replaced by
a $\delta$-mock-Novikov multiplication.
\begin{definition}
A triple $(\mathcal{A},\ast,\circ)$ is called a
\emph{$\delta$-mock-Novikov-Poisson algebra} if $(\mathcal{A},\ast)$ is an
anti-commutative anti-associative algebra, $(\mathcal{A},\circ)$ is a
$\delta$-mock-Novikov algebra, and
\begin{align}
(x\ast y)\circ z+x\ast(y\circ z)&=0,
\label{id: NP  ass}\\
\delta(x\circ y)\ast z+\delta(y\circ x)\ast z
+x\circ(y\ast z)+y\circ(x\ast z)&=0
\label{mockNP}
\end{align}
for all $x,y,z\in \mathcal{A}$.
\end{definition}

The differential construction considered in Section \ref{sec2}
provides a natural source of such algebras. Namely, if
$(\mathcal{A},\ast,\varphi)$ is a $\delta$-differential anti-commutative
anti-associative algebra, then the multiplication
$x\circ y=x\ast\varphi(y)$ makes $(\mathcal{A},\ast,\circ)$ a
$\delta$-mock-Novikov-Poisson algebra.

We next relate the two operations of a
$\delta$-mock-Novikov-Poisson algebra to a single
$\delta$-mock-Novikov multiplication. The following result shows that
they form a compatible structure.
\begin{proposition}\label{prop:depolar-mNP}
Let $(\mathcal{A},\ast,\circ)$ be a
$\delta$-mock-Novikov-Poisson algebra. For any scalars
$\alpha,\beta\in\mathbb C$, define
\[
x\cdot_{\alpha,\beta}y
=
\alpha\,x\circ y+\beta\,x\ast y.
\]
Then $(\mathcal{A},\cdot_{\alpha,\beta})$ is a
$\delta$-mock-Novikov algebra. In particular,
$x\cdot y=x\circ y+x\ast y$
is called the \emph{depolarization} of
$(\mathcal{A},\ast,\circ)$.
\end{proposition}

\begin{proof}
By Proposition~\ref{prop:anti-mN},
$(\mathcal{A},\ast)$ is itself a $\delta$-mock-Novikov algebra.
Hence the pure $\alpha^2$- and $\beta^2$-terms in the defining
identities vanish. 

It remains only to check the mixed terms. Using
\eqref{id: NP  ass}, \eqref{mockNP}, their permutations, and the
anti-commutativity and anti-associativity of $\ast$, a direct
calculation shows that the mixed terms in both
\eqref{deltaMN2} and \eqref{deltaMN1} vanish. Hence
$(\mathcal{A},\cdot_{\alpha,\beta})$ is a
$\delta$-mock-Novikov algebra.
\end{proof}

As an immediate consequence, the depolarized multiplication remains
compatible with the original anti-commutative anti-associative
multiplication.
\begin{corollary}\label{cor:pencil-mNP}
Let $(\mathcal{A},\ast,\circ)$ be a
$\delta$-mock-Novikov-Poisson algebra. For any scalars
$\alpha,\beta\in\mathbb C$, define
\[
x\times_{\alpha,\beta}y
=
\alpha\,x\circ y+\beta\,x\ast y.
\]
Then
$(\mathcal{A},\ast,\times_{\alpha,\beta})$
is a $\delta$-mock-Novikov-Poisson algebra.
\end{corollary}

It is therefore natural to ask for the converse: when does the
canonical polarization of a $\delta$-mock-Novikov multiplication
recover a $\delta$-mock-Novikov-Poisson structure? The following proposition gives a necessary and sufficient condition.

\begin{proposition}\label{pol-mNP}
Let $(\mathcal{A},\cdot)$ be a $\delta$-mock-Novikov algebra over a field of
characteristic different from $2$. For all $x,y,z\in \mathcal{A}$, define
\[
x\ast y=\frac{xy-yx}{2},
\qquad
x\circ y=\frac{xy+yx}{2}.
\]
Then $(\mathcal{A},\ast,\circ)$ is a $\delta$-mock-Novikov-Poisson algebra if
and only if
\begin{equation}\label{pol-eq}
x(yz)-z(yx)-(1+\delta)(xz)y+(1-\delta)(zx)y=0.
\end{equation}

\end{proposition}

\begin{proof}
For all $x,y,z\in \mathcal{A}$, we have
\[
\begin{aligned}
4\big((x\ast y)\circ z+x\ast(y\circ z)\big)
={}&(xy)z-(yx)z+z(xy)-z(yx)+x(yz)+x(zy)-(yz)x-(zy)x\\
\stackrel{\eqref{deltaMN2}}{=}&-(xz)y+z(xy)-z(yx)
+x(yz)+x(zy)+(zx)y\\
\stackrel{\eqref{deltaMN1}}{=}&x(yz)-z(yx)
-(1+\delta)(xz)y+(1-\delta)(zx)y,
\end{aligned}
\]
where in the last equality \eqref{deltaMN1} is applied to $(x,z,y)$.
Thus \eqref{pol-eq} is equivalent to \eqref{id: NP  ass}.

The necessity is immediate. Conversely, assume that \eqref{pol-eq}
holds. Then \eqref{id: NP  ass} holds. Put
$B=(x\ast y)\circ z$. From \eqref{id: NP  ass} and its permutations
we obtain both $z\ast(x\circ y)=B$ and
$z\ast(x\circ y)=-B$. Hence $B=0$, and therefore
\begin{equation}\label{zero}
(x\ast y)\circ z=x\ast(y\circ z)
=(x\circ y)\ast z=x\circ(y\ast z)=0.
\end{equation}

Since $xy=x\ast y+x\circ y$, identities \eqref{deltaMN2} and
\eqref{zero} give
\begin{equation}\label{eqq1}
(x\ast y)\ast z+(x\ast z)\ast y
+(x\circ y)\circ z+(x\circ z)\circ y=0,
\end{equation}
whereas \eqref{deltaMN1} gives
\begin{equation}\label{eqq2}
2\delta(x\circ y)\circ z
+x\ast(y\ast z)+y\ast(x\ast z)
+x\circ(y\circ z)+y\circ(x\circ z)=0.
\end{equation}
Applying \eqref{eqq1} to $(z,x,y)$ and using the
anti-commutativity of $\ast$ and the commutativity of $\circ$, we get
\[
x\ast(y\ast z)+y\ast(x\ast z)
+x\circ(y\circ z)+y\circ(x\circ z)=0.
\]
Together with \eqref{eqq2} and $\delta\neq0$, this implies
$(x\circ y)\circ z=0$, and hence also $x\circ(y\circ z)=0$.
Thus $(\mathcal{A},\circ)$ is a $\delta$-mock-Novikov algebra.

Now \eqref{eqq1} reduces to
$(x\ast y)\ast z+(x\ast z)\ast y=0$. Interchanging $x$ and $y$
and using anti-commutativity yields
\[
(x\ast y)\ast z+x\ast(y\ast z)=0,
\]
so $(\mathcal{A},\ast)$ is anti-associative. Finally,
\eqref{id: NP  ass} holds by assumption and \eqref{mockNP} follows
from \eqref{zero}. Hence $(\mathcal{A},\ast,\circ)$ is a
$\delta$-mock-Novikov-Poisson algebra.
\end{proof}

Recall that the Kantor product of two bilinear multiplications
$A$ and $B$, with respect to a fixed element $u$, is defined by
\begin{equation}\label{Kantor-product}
K_u(A,B)(x,y)
=
A(u,B(x,y))
-B(A(u,x),y)
-B(x,A(u,y)).
\end{equation}
For $\delta$-Novikov-Poisson algebras, the Kantor product of the
associative and Novikov multiplications is again a
$\delta$-Novikov multiplication; see~\cite[Theorem~26]{K}.
The following result gives the corresponding statement in the
mock setting.

\begin{theorem}\label{thm:Kantor-mNP}
Let $(\mathcal{A},\ast,\circ)$ be a
$\delta$-mock-Novikov-Poisson algebra, and fix $u\in\mathcal{A}$.
Let $\kappa_u=K_u(\ast,\circ)$ be the Kantor product
with respect to $u$ defined by~\eqref{Kantor-product}.
Then $(\mathcal{A},\kappa_u)$ is a
$\delta$-mock-Novikov algebra. Moreover,
\[
\kappa_u\bigl(\kappa_u(x,y),z\bigr)=0
\]
for all $x,y,z\in\mathcal{A}$.
\end{theorem}
\begin{proof}
By~\eqref{id: NP  ass},
we get 
$(u\ast x)\circ y=-u\ast(x\circ y),$
and hence
\begin{equation}\label{eq:Kantor-simplified}
\kappa_u(x,y)
=
2u\ast(x\circ y)-x\circ(u\ast y).
\end{equation}

We first show that
$\kappa_u\bigl(\kappa_u(x,y),z\bigr)=0.$
Expanding by~\eqref{eq:Kantor-simplified} and using
anti-associativity, \eqref{id: NP  ass}, and
\eqref{deltaMN2}, we obtain
\[
\begin{aligned}
2\kappa_u\bigl(\kappa_u(x,y),z\bigr)
={}&
2\Bigl(
\delta(u\circ x)\ast y
+\delta(x\circ u)\ast y
+u\circ(x\ast y)
+x\circ(u\ast y)
\Bigr)\circ(u\ast z)\\
&+
\Bigl(
2\delta(u\circ u)\ast z
+2u\circ(u\ast z)
\Bigr)\circ(x\ast y).
\end{aligned}
\]
The first parenthesis vanishes by~\eqref{mockNP} applied to
$(u,x,y)$, while the second vanishes by~\eqref{mockNP} applied to
$(u,u,z)$. Since the characteristic is different from $2$,
$\kappa_u\bigl(\kappa_u(x,y),z\bigr)=0.$
Therefore
\[
\kappa_u\bigl(\kappa_u(x,y),z\bigr)
+
\kappa_u\bigl(\kappa_u(x,z),y\bigr)=0,
\]
so $\kappa_u$ satisfies the right-anti-commutative identity.

It remains to verify the left $\delta$-mock-pre-Lie identity.
In view of the preceding identity, it suffices to prove
\[
\kappa_u\bigl(x,\kappa_u(y,z)\bigr)
+
\kappa_u\bigl(y,\kappa_u(x,z)\bigr)=0.
\]
A direct expansion using~\eqref{eq:Kantor-simplified}, followed by
applications of
\eqref{deltaMN1}, \eqref{deltaMN2}, and
\eqref{id: NP  ass}, gives
\[
\begin{aligned}
&\kappa_u\bigl(x,\kappa_u(y,z)\bigr)
+\kappa_u\bigl(y,\kappa_u(x,z)\bigr)\\
={}&\delta\Bigl[
\delta(u\circ y)\ast\bigl(x\circ(u\ast z)\bigr)
+\delta(y\circ u)\ast\bigl(x\circ(u\ast z)\bigr)
+u\circ\bigl(y\ast(x\circ(u\ast z))\bigr)
+y\circ\bigl(u\ast(x\circ(u\ast z))\bigr)\\
&+\delta\bigl(y\circ(x\circ(u\ast z))\bigr)\ast u
+\delta\bigl((x\circ(u\ast z))\circ y\bigr)\ast u
+y\circ\bigl((x\circ(u\ast z))\ast u\bigr)
+(x\circ(u\ast z))\circ(y\ast u)\\
&-2\delta\bigl(u\circ(y\circ(u\ast z))\bigr)\ast x
-2\delta\bigl((y\circ(u\ast z))\circ u\bigr)\ast x
-2u\circ\bigl((y\circ(u\ast z))\ast x\bigr)
-2(y\circ(u\ast z))\circ(u\ast x)\\
&-\delta\bigl(u\circ(x\circ(u\ast z))\bigr)\ast y
-\delta\bigl((x\circ(u\ast z))\circ u\bigr)\ast y
-u\circ\bigl((x\circ(u\ast z))\ast y\bigr)
-(x\circ(u\ast z))\circ(u\ast y)
\Bigr].
\end{aligned}
\]
The four groups of terms above vanish, respectively, by
\eqref{mockNP} applied to
\[
(u,y,x\circ(u\ast z)),\qquad
(y,x\circ(u\ast z),u),\qquad
(u,y\circ(u\ast z),x),\qquad
(u,x\circ(u\ast z),y).
\]
Therefore
\[
\kappa_u\bigl(x,\kappa_u(y,z)\bigr)
+\kappa_u\bigl(y,\kappa_u(x,z)\bigr)=0.
\]
Since
\[
\kappa_u\bigl(\kappa_u(a,b),c\bigr)=0
\qquad(a,b,c\in\mathcal A),
\]
the left $\delta$-mock-pre-Lie identity follows immediately.
In conjunction with the right anti-commutative identity established above,
this proves that $(\mathcal A,\kappa_u)$ is a
$\delta$-mock-Novikov algebra.
\end{proof}

We next introduce mock analogues of Poisson and transposed Poisson
algebras, with the main aim of relating the latter to
$\delta$-mock-Novikov-Poisson algebras. 

\begin{definition}
Let $(\mathcal{A},\ast,\{-,-\})$ be a vector space such that
$(\mathcal{A},\ast)$ is an anti-commutative anti-associative algebra
and $(\mathcal{A},\{-,-\})$ is a mock-Lie algebra. It is called a
\emph{$\delta$-mock-Poisson algebra} if
\begin{equation}\label{deltaMP}
\{x\ast y,z\}
=
-\delta\bigl(x\ast\{y,z\}+\{x,z\}\ast y\bigr),
\end{equation}
and a \emph{transposed $\delta$-mock-Poisson algebra} if
\begin{equation}\label{deltaTMP}
-2\delta\,x\ast\{y,z\}
=
\{x\ast y,z\}+\{y,x\ast z\},
\end{equation}
for all $x,y,z\in\mathcal{A}$.
\end{definition}

For $\delta=1$, every mock-Poisson algebra is, in particular, an
anti-Poisson algebra~\cite{R22}.

A natural class of $\delta$-mock-Poisson algebras can be obtained
from commuting $\delta$-derivations, in parallel with the classical
Poisson construction.

\begin{example}
Let $(\mathcal{A},\ast)$ be an anti-commutative anti-associative algebra, and
let $D_1,\ldots,D_n$ be pairwise commuting $\delta$-derivations of
$\mathcal{A}$. Define
\[
\{x,y\}
=
\sum_{1\leq i<j\leq n}
\bigl(D_i(x)\ast D_j(y)+D_i(y)\ast D_j(x)\bigr).
\]
Then $(\mathcal{A},\{-,-\})$ is a mock-Lie algebra and
$(\mathcal{A},\ast,\{-,-\})$ is a $\delta$-mock-Poisson algebra.
\end{example}

The relevance of the transposed notion is illustrated by the following
result, which provides the mock counterpart of the classical
connection between Novikov-Poisson and transposed Poisson
structures.
\begin{theorem}\label{thm:mNP-to-tMP}
Let $(\mathcal{A},\ast,\circ)$ be a
$(2\delta-1)$-mock-Novikov-Poisson algebra. Define
\[
\{x,y\}=x\circ y+y\circ x.
\]
Then $(\mathcal{A},\ast,\{-,-\})$ is a transposed
$\delta$-mock-Poisson algebra.
\end{theorem}

\begin{proof}
Since $(\mathcal{A},\circ)$ is a $(2\delta-1)$-mock-Novikov algebra,
Proposition~\ref{bdnov} shows that $(\mathcal{A},\{-,-\})$ is a mock-Lie
algebra. It remains to verify \eqref{deltaTMP}.

Put $\lambda=2\delta-1$. By \eqref{id: NP  ass},
\[
(x\ast y)\circ z=-x\ast(y\circ z),
\qquad
(x\ast z)\circ y=-x\ast(z\circ y).
\]
Moreover, applying \eqref{mockNP} with parameter $\lambda$ to
$(y,z,x)$ gives
\[
z\circ(x\ast y)+y\circ(x\ast z)
=-\lambda\,x\ast\{y,z\}.
\]
Therefore,
\[
\begin{aligned}
\{x\ast y,z\}+\{y,x\ast z\}
={}&(x\ast y)\circ z+z\circ(x\ast y)+y\circ(x\ast z)+(x\ast z)\circ y\\
\stackrel{\eqref{id: NP  ass}}{=}{}&-(1+\lambda)x\ast\{y,z\}\\
={}&-2\delta\,x\ast\{y,z\}.
\end{aligned}
\]
Thus \eqref{deltaTMP} holds.
\end{proof}

Theorem~\ref{thm:mNP-to-tMP}, together with the differential
construction above, yields the following class of transposed
$\delta$-mock-Poisson algebras.

\begin{corollary}
Let $(\mathcal{A},\ast)$ be an anti-commutative anti-associative algebra and
let $D$ be a $(2\delta-1)$-derivation of $\mathcal{A}$. Define
\[
\{x,y\}
=
x\ast D(y)+y\ast D(x).
\]
Then $(\mathcal{A},\ast,\{-,-\})$ is a transposed
$\delta$-mock-Poisson algebra.
\end{corollary}

\begin{proof}
Define $x\circ y=x\ast D(y)$. By the differential construction for
mock-Novikov-Poisson algebras, $(\mathcal{A},\ast,\circ)$ is a
$(2\delta-1)$-mock-Novikov-Poisson algebra. The conclusion now
follows from Theorem~\ref{thm:mNP-to-tMP}.
\end{proof}

We next turn to Gelfand-Dorfman type structures. A basic construction
associates a Gelfand-Dorfman algebra with every Novikov algebra by
taking the commutator
\[
[x,y]=x\circ y-y\circ x;
\]
see, for example,~\cite{GelDorf79}. Moreover, Sartayev~\cite{SartayevTP}
showed that transposed Poisson algebras coincide with
Gelfand-Dorfman algebras with commutative Novikov multiplication.
These two connections motivate the following mock analogue.

\begin{definition}
A triple $(\mathcal{A},\circ,\{-,-\})$ is called a
\emph{$\delta$-mock-Gelfand-Dorfman algebra} if $(\mathcal{A},\circ)$ is a
$\delta$-mock-Novikov algebra, $(\mathcal{A},\{-,-\})$ is a mock-Lie algebra,
and
\begin{equation}\label{id:GD}
\{x,y\circ z\}+\{z,y\circ x\}
+\delta\{y,x\}\circ z+\delta\{y,z\}\circ x
+y\circ\{x,z\}=0
\end{equation}
for all $x,y,z\in \mathcal{A}$.
\end{definition}

The classical commutator construction for Novikov algebras is replaced
in the mock setting by symmetrization.

\begin{proposition}\label{prop:mN-to-mGD}
Let $(\mathcal{A},\circ)$ be a $\delta$-mock-Novikov algebra. Define
\[
\{x,y\}=x\circ y+y\circ x.
\]
Then $(\mathcal{A},\circ,\{-,-\})$ is a
$\delta$-mock-Gelfand-Dorfman algebra.
\end{proposition}

\begin{proof}
By Proposition~\ref{bdnov}, $(\mathcal{A},\{-,-\})$ is a mock-Lie algebra.
Expanding the left-hand side of \eqref{id:GD} gives
\[
\begin{aligned}
&\{x,y\circ z\}+\{z,y\circ x\}
+\delta\{y,x\}\circ z+\delta\{y,z\}\circ x
+y\circ\{x,z\}\\
={}&
\Big(
\delta(x\circ y)\circ z+\delta(y\circ x)\circ z
+x\circ(y\circ z)+y\circ(x\circ z)
\Big)\\
&+
\Big(
\delta(y\circ z)\circ x+\delta(z\circ y)\circ x
+y\circ(z\circ x)+z\circ(y\circ x)
\Big)\\
&+(y\circ z)\circ x+(y\circ x)\circ z.
\end{aligned}
\]
The first two terms vanish by \eqref{deltaMN1}, while the last term
vanishes by \eqref{deltaMN2}. Hence \eqref{id:GD} holds.
\end{proof}

We next establish the mock counterpart of the classical equality
between commutative Gelfand-Dorfman algebras and transposed Poisson
algebras. 

For later use, set $\lambda=2\delta-1$ and define
\begin{align*}
\mathrm{GD}_{\lambda}(x,y,z)
={}&\{x,y\ast z\}+\{z,y\ast x\}
+\lambda\{y,x\}\ast z
+\lambda\{y,z\}\ast x+y\ast\{x,z\},
\\
\mathcal T_{\delta}(x,y,z)
={}&2\delta\,x\ast\{y,z\}
+\{x\ast y,z\}+\{y,x\ast z\}.
\end{align*}

\begin{theorem}\label{thm:mGD-tMP}
Let $\delta\in\mathbb C^\times$ and set $\lambda=2\delta-1$.
Assume that $\lambda\neq0$.
\begin{enumerate}[\rm (i)]
\item Every transposed $\delta$-mock-Poisson algebra
$(\mathcal{A},\ast,\{-,-\})$ is an anti-commutative
$\lambda$-mock-Gelfand-Dorfman algebra.

\item If $\delta\neq\frac34$, then every anti-commutative
$\lambda$-mock-Gelfand-Dorfman algebra
$(\mathcal{A},\ast,\{-,-\})$ is a transposed
$\delta$-mock-Poisson algebra.
\end{enumerate}
Consequently, for
$\delta\notin\{\frac12,\frac34\}$, the varieties of transposed
$\delta$-mock-Poisson algebras and anti-commutative
$(2\delta-1)$-mock-Gelfand-Dorfman algebras coincide.
\end{theorem}

\begin{proof}
We first prove \textup{(i)}. Let $(\mathcal{A},\ast,\{-,-\})$ be a transposed
$\delta$-mock-Poisson algebra. Since $(\mathcal{A},\ast)$ is anti-commutative
and anti-associative, Proposition~\ref{prop:anti-mN} implies that
$(\mathcal{A},\ast)$ is a $\lambda$-mock-Novikov algebra. Moreover,
$(\mathcal{A},\{-,-\})$ is a mock-Lie algebra.

Taking the cyclic sum of
$\mathcal T_{\delta}(x,y,z)=0$ and using the
anti-commutativity of $\ast$ and the symmetry of $\{-,-\}$ gives
\begin{equation}\label{cyclicTMP-GD}
x\ast\{y,z\}+y\ast\{z,x\}+z\ast\{x,y\}=0,
\end{equation}
since $\delta\neq0$.

A direct calculation gives
\begin{align*}
\mathrm{GD}_{\lambda}(x,y,z)
={}&\frac12\Big(
-\mathcal T_{\delta}(x,y,z)
+\mathcal T_{\delta}(y,x,z)
-\mathcal T_{\delta}(z,y,x)
\Big)\notag\\
&+(1-\delta)
\Big(
x\ast\{y,z\}+y\ast\{z,x\}+z\ast\{x,y\}
\Big).
\end{align*}
Hence \eqref{cyclicTMP-GD} and
$\mathcal T_{\delta}=0$ imply
$\mathrm{GD}_{\lambda}=0$. Thus
$(\mathcal{A},\ast,\{-,-\})$ is a
$\lambda$-mock-Gelfand-Dorfman algebra.

Conversely, suppose that $(\mathcal{A},\ast,\{-,-\})$ is an anti-commutative
$\lambda$-mock-Gelfand-Dorfman algebra. By
Proposition~\ref{prop:anti-mN}, $(\mathcal{A},\ast)$ is anti-associative, while
$(\mathcal{A},\{-,-\})$ is a mock-Lie algebra. It therefore remains only to
verify the transposed compatibility condition.

For $\lambda=2\delta-1$, a direct calculation yields
\begin{align}
(3-4\delta)\mathcal T_{\delta}(x,y,z)
={}&(2\delta-1)\mathrm{GD}_{\lambda}(z,y,x)
 +(2\delta-1)\mathrm{GD}_{\lambda}(y,z,x)\notag\\
&+2(1-\delta)\mathrm{GD}_{\lambda}(z,x,y).
\label{TPfromGD}
\end{align}
Since $\mathrm{GD}_{\lambda}=0$ and
$\delta\neq\frac34$, it follows that
$\mathcal T_{\delta}=0$. Thus
$(\mathcal{A},\ast,\{-,-\})$ is a transposed
$\delta$-mock-Poisson algebra.
\end{proof}

For brevity, we write
\[
\mathrm{m}\mathcal{N}_\lambda,\qquad
\mathrm{m}\mathcal{NP}_\lambda,\qquad
\mathrm{m}\mathcal{P}_\delta,\qquad
\mathrm{tm}\mathcal{P}_\delta,\qquad
\mathrm{m}\mathcal{GD}_\lambda
\]
for $\lambda$-mock-Novikov, $\lambda$-mock-Novikov-Poisson,
$\delta$-mock-Poisson, transposed $\delta$-mock-Poisson, and
$\lambda$-mock-Gelfand-Dorfman algebras, respectively.

Setting $\lambda=2\delta-1$, the main constructions and
correspondences established above may be summarized schematically as
follows:
\begin{equation}\label{diag:mock-structures}
\begin{tikzcd}[column sep=huge,row sep=2cm]
\mathrm{m}\mathcal{N}_{\lambda}
\arrow[r, "\text{\rm Prop.~\ref{prop:mN-to-mGD}}"]
\arrow[d, dashed,
       "\text{\rm Prop.~\ref{pol-mNP}}" description]
&
\mathrm{m}\mathcal{GD}_{\lambda}
\arrow[d, leftrightarrow,
       "\text{\rm Thm.~\ref{thm:mGD-tMP}}"]
\\
\mathrm{m}\mathcal{NP}_{\lambda}
\arrow[u, bend left=45,
       "\text{\rm Prop.~\ref{prop:depolar-mNP}}"{near start}]
\arrow[u, bend right=45,
       "\text{\rm Thm.~\ref{thm:Kantor-mNP}}"' {near end}]
\arrow[r, "\text{\rm Thm.~\ref{thm:mNP-to-tMP}}"']
&
\mathrm{tm}\mathcal{P}_{\delta}.
\end{tikzcd}
\end{equation}

\subsection{Polarization, depolarization, and tensor products}
Having established the polarization and depolarization between
$\delta$-mock-Novikov and $\delta$-mock-Novikov-Poisson algebras in
Propositions~\ref{pol-mNP} and~\ref{prop:depolar-mNP}, we now turn to
the corresponding one-operation descriptions of
$\delta$-mock-Poisson and transposed $\delta$-mock-Poisson algebras.

Recall that Markl and Remm~\cite{MarklRemm} characterized Poisson
algebras in terms of a single nonassociative multiplication through
polarization and depolarization. We first establish the corresponding
result for $\delta$-mock-Poisson algebras.

\begin{definition}\label{def:delta-mP-admissible}
An algebra $(\mathcal{A},\cdot)$ is called an
\emph{admissible $\delta$-mock-Poisson algebra} if for all $x,y,z\in \mathcal{A}$:
\begin{equation}\label{delta-admMP}
\begin{aligned}
f_\delta(x,y,z):={}&
(1+2\delta)(xy)z-(yx)z+(xz)y-\delta(yz)x\\
&+(2\delta-1)(zx)y+3\delta\,x(yz)=0.
\end{aligned}
\end{equation}

\end{definition}

The following result shows that Definition~\ref{def:delta-mP-admissible}
is precisely the one-operation counterpart of
$\delta$-mock-Poisson algebras.

\begin{theorem}
Let $(\mathcal{A},\cdot)$ be an algebra.
Define
\[
x\ast y=\frac12(xy-yx),
\qquad
\{x,y\}=\frac12(xy+yx).
\]
Then $(\mathcal{A},\cdot)$ is an admissible $\delta$-mock-Poisson algebra if
and only if $(\mathcal{A},\ast,\{-,-\})$ is a
$\delta$-mock-Poisson algebra.
\end{theorem}

\begin{proof}
The anti-commutativity of $\ast$ and the symmetry of $\{-,-\}$ are
automatic. Set
\begin{align*}
P_1(x,y,z)
&=4\bigl((x\ast y)\ast z+x\ast(y\ast z)\bigr),\\
P_2(x,y,z)
&=4\bigl(
\{\{x,y\},z\}
+\{\{y,z\},x\}
+\{\{z,x\},y\}
\bigr),\\
P_3^\delta(x,y,z)
&=4\Bigl(
\{x\ast y,z\}
+\delta\,x\ast\{y,z\}
+\delta\,\{x,z\}\ast y
\Bigr).
\end{align*}
Thus $P_1=0$, $P_2=0$, and $P_3^\delta=0$ are precisely the
anti-associativity of $\ast$, the mock-Jacobi identity for
$\{-,-\}$, and the $\delta$-mock-Poisson compatibility
\eqref{deltaMP}, respectively.

A direct expansion gives
\begin{align*}
3\delta P_1(x,y,z)
={}&f_\delta(x,y,z)+f_\delta(z,y,x)
-f_\delta(x,z,y)-f_\delta(z,x,y),\\
3\delta P_2(x,y,z)
={}&f_\delta(x,y,z)+f_\delta(y,x,z)+f_\delta(z,y,x)+f_\delta(x,z,y)+f_\delta(y,z,x)+f_\delta(z,x,y),\\
3\delta P_3^\delta(x,y,z)
={}&\delta f_\delta(x,y,z)-\delta f_\delta(y,x,z)
-f_\delta(z,y,x)\notag+\delta f_\delta(x,z,y)-\delta f_\delta(y,z,x)
+f_\delta(z,x,y).
\end{align*}
Hence \eqref{delta-admMP} implies
$P_1=P_2=P_3^\delta=0$, since $\delta\neq0$. Therefore the
polarization of $(\mathcal{A},\cdot)$ is a $\delta$-mock-Poisson algebra.

Conversely, one has
\begin{align*}
2f_\delta(x,y,z)
={}&(\delta-1)P_3^\delta(z,x,y)
+\delta P_3^\delta(y,z,x)
-(\delta+1)P_3^\delta(y,x,z)\notag\\
&-2\delta P_1(y,z,x)
-(\delta+1)P_1(z,x,y)
+\delta P_2(z,x,y).
\end{align*}
Thus, if $(\mathcal{A},\ast,\{-,-\})$ is a $\delta$-mock-Poisson algebra,
then $P_1=P_2=P_3^\delta=0$, and hence
$f_\delta(x,y,z)=0$. Therefore $(\mathcal{A},\cdot)$ is an admissible
$\delta$-mock-Poisson algebra.
\end{proof}

We next consider the one-operation form of transposed
$\delta$-mock-Poisson algebras. Recall that weak Leibniz algebras are
related to transposed Poisson algebras through polarization and
depolarization~\cite{DzhWeakLeibniz}. On the mock side, anti-Leibniz
algebras \cite{BCKM} may be viewed as the counterpart of Leibniz algebras. We therefore introduce the following
concept.

\begin{definition}
An algebra $(\mathcal{A},\cdot)$ is called a
\emph{$\delta$-weak anti-Leibniz algebra} if
\begin{align}
&\delta\bigl((xy)z+(yx)z\bigr)
 +(\delta+1)\bigl(x(yz)+y(xz)\bigr) 
 +(\delta-1)\bigl(x(zy)+y(zx)\bigr)=0,
\label{deltaMWL1}\\
&(\delta+1)\bigl((xy)z+(xz)y\bigr)
 +(\delta-1)\bigl((yx)z+(zx)y\bigr) 
 +\delta\,x(yz+zy)=0
\label{deltaMWL2}
\end{align}
for all $x,y,z\in \mathcal{A}$.
\end{definition}
\begin{remark}
When $\delta=1$, every two-sided anti-Leibniz algebra is a weak
anti-Leibniz algebra in the above sense. Indeed, symmetrizing the left anti-Leibniz identity in $x,y$ and the right anti-Leibniz identity in $y,z$ yields~\eqref{deltaMWL1} and~\eqref{deltaMWL2}, respectively.
\end{remark}

The terminology is justified by the following polarization-depolarization
correspondence.

\begin{theorem}
Let $(\mathcal{A},\cdot)$ be an algebra.
Define
\[
x\ast y=\frac{xy-yx}{2},
\qquad
\{x,y\}=\frac{xy+yx}{2}.
\]
Then $(\mathcal{A},\cdot)$ is a $\delta$-weak anti-Leibniz algebra if and only if
$(\mathcal{A},\ast,\{-,-\})$ is a transposed $\delta$-mock-Poisson algebra.
\end{theorem}

\begin{proof}
The anti-commutativity of $\ast$ and the symmetry of $\{-,-\}$ are
automatic. Set
\begin{align*}
P_1(x,y,z)
&=4\bigl((x\ast y)\ast z+x\ast(y\ast z)\bigr),\\
P_2(x,y,z)
&=4\bigl(
\{\{x,y\},z\}
+\{\{y,z\},x\}
+\{\{z,x\},y\}
\bigr),\\
P_3(x,y,z)
&=4\bigl(
2\delta\,x\ast\{y,z\}
+\{x\ast y,z\}
+\{y,x\ast z\}
\bigr).
\end{align*}
Thus $P_1=0$, $P_2=0$, and $P_3=0$ are precisely the
anti-associativity of $\ast$, the mock-Jacobi identity for
$\{-,-\}$, and the transposed $\delta$-mock-Poisson compatibility
condition \eqref{deltaTMP}, respectively.

For convenience, denote the left-hand sides of
\eqref{deltaMWL1} and \eqref{deltaMWL2} by
$F_\delta(x,y,z)$ and $G_\delta(x,y,z)$, respectively. A direct
expansion of the polarized operations gives
\begin{align*}
F_\delta(x,y,z)
={}&-\frac12 P_1(x,z,y)
+\frac{\delta}{2}P_2(x,y,z) +\frac14\Bigl(
P_3(x,y,z)+P_3(y,x,z)-P_3(z,x,y)
\Bigr),\\
G_\delta(x,y,z)
={}&\frac12P_1(x,y,z)+\frac12P_1(x,z,y)
+\frac{\delta}{2}P_2(x,y,z) +\frac14\Bigl(
P_3(x,y,z)-P_3(y,x,z)-P_3(z,x,y)
\Bigr).
\end{align*}
Hence every transposed $\delta$-mock-Poisson algebra depolarizes to a
$\delta$-weak anti-Leibniz algebra.

Conversely, the same computation yields
\begin{align*}
P_1(x,y,z)
={}&\frac13\Bigl(
F_\delta(x,y,z)-2F_\delta(x,z,y)+F_\delta(y,z,x) \notag\\
&\hspace{15mm}
+G_\delta(x,y,z)-2G_\delta(y,x,z)+G_\delta(z,x,y)
\Bigr),\\
P_2(x,y,z)
={}&\frac{1}{3\delta}\Bigl(
F_\delta(x,y,z)+F_\delta(x,z,y)+F_\delta(y,z,x) \notag\\
&\hspace{15mm}
+G_\delta(x,y,z)+G_\delta(y,x,z)+G_\delta(z,x,y)
\Bigr),\\
P_3(x,y,z)
={}&F_\delta(x,y,z)+F_\delta(x,z,y) -G_\delta(y,x,z)-G_\delta(z,x,y).
\end{align*}
Therefore \eqref{deltaMWL1} and \eqref{deltaMWL2} imply
$P_1=P_2=P_3=0$. Thus $(\mathcal{A},\ast)$ is anti-associative,
$(\mathcal{A},\{-,-\})$ is a mock-Lie algebra, and \eqref{deltaTMP} holds.
Hence $(\mathcal{A},\ast,\{-,-\})$ is a transposed
$\delta$-mock-Poisson algebra.
\end{proof}

We now consider tensor products of the Poisson-type structures
introduced above. For the corresponding non-mock classes, we write
$\mathcal{NP}_\delta,\;
\mathcal{P}_\delta,\;
\mathrm{t}\mathcal{P}_\delta$
for the classes of $\delta$-Novikov-Poisson, $\delta$-Poisson, and
transposed $\delta$-Poisson algebras, respectively; see~\cite{K}.
For $\delta=1$, these structures are closed under their natural
tensor products \cite{Xu96,TP1,LB}.
By contrast, such tensor-closure fails for the mock structures
considered here, as well as for $F$-manifold algebras and generalized
Poisson algebras; see~\cite{LSB,zus}.

Accordingly, consider the three pairs
\[
(\mathsf{X}_\delta,\mathsf{Y}_\delta)
\in
\big\{
(\mathcal{NP}_\delta,\mathrm{m}\mathcal{NP}_\delta),
(\mathcal{P}_\delta,\mathrm{m}\mathcal{P}_\delta),
(\mathrm{t}\mathcal{P}_\delta,\mathrm{tm}\mathcal{P}_\delta)
\big\}.
\]
Fix one such pair, and for $i=1,2$, let
$(\mathcal{A}_i,\ast_i,\diamond_i)$ be either an
$\mathsf{X}_\delta$-algebra or a $\mathsf{Y}_\delta$-algebra.
Set
\[
\varepsilon_i=
\begin{cases}
1, & \text{if $\mathcal{A}_i$ is an
$\mathsf{X}_\delta$-algebra},\\
-1, & \text{if $\mathcal{A}_i$ is a
$\mathsf{Y}_\delta$-algebra}.
\end{cases}
\]
Here $\diamond_i$ denotes $\circ_i$ in the Novikov-Poisson case and
$\{-,-\}_i$ in the Poisson and transposed Poisson cases.

Define bilinear operations $\ast$ and $\diamond$ on
$\mathcal A_1\otimes\mathcal A_2$ by
\begin{equation}\label{tensor-product}
\begin{aligned}
(x_1\otimes x_2)\ast(y_1\otimes y_2)
&=
(x_1\ast_1 y_1)\otimes(x_2\ast_2 y_2),\\
(x_1\otimes x_2)\diamond(y_1\otimes y_2)
&=
(x_1\diamond_1 y_1)\otimes(x_2\ast_2 y_2)
+(x_1\ast_1 y_1)\otimes(x_2\diamond_2 y_2),
\end{aligned}
\end{equation}
for $x_i,y_i\in\mathcal A_i$, and extend them bilinearly.

\begin{theorem}\label{thm:tensor-unified}
With the notation above and the operations defined by
\eqref{tensor-product},
$(\mathcal{A}_1\otimes\mathcal{A}_2,\ast,\diamond)$ is an
$\mathsf{X}_\delta$-algebra if
$\varepsilon_1\varepsilon_2=1$, and a
$\mathsf{Y}_\delta$-algebra if
$\varepsilon_1\varepsilon_2=-1$.
\end{theorem}

\begin{proof}
Write
$x=x_1\otimes x_2,\;
y=y_1\otimes y_2,\;
z=z_1\otimes z_2,$
and set $\varepsilon=\varepsilon_1\varepsilon_2$.
For each factor, the multiplication $\ast_i$ satisfies
$a\ast_i b=\varepsilon_i b\ast_i a$ and
$(a\ast_i b)\ast_i c=\varepsilon_i a\ast_i(b\ast_i c)$.
It follows immediately from~\eqref{tensor-product} that
\[
x\ast y=\varepsilon\,y\ast x,
\qquad
(x\ast y)\ast z=\varepsilon\,x\ast(y\ast z).
\]
Thus $\ast$ is commutative associative if $\varepsilon=1$, and
anti-commutative anti-associative if $\varepsilon=-1$.

We first consider
\[
(\mathsf X_\delta,\mathsf Y_\delta)
=
(\mathcal{NP}_\delta,\mathrm{m}\mathcal{NP}_\delta).
\]
The defining identities of the two classes can be written uniformly,
for $i=1,2$, as
\begin{align}
(a\diamond_i b)\diamond_i c
&=\varepsilon_i(a\diamond_i c)\diamond_i b,
\label{eq:tensor-unified-NP1}\\
\delta(a\diamond_i b)\diamond_i c
-\varepsilon_i a\diamond_i(b\diamond_i c)
&=
\varepsilon_i\delta(b\diamond_i a)\diamond_i c
-b\diamond_i(a\diamond_i c),
\label{eq:tensor-unified-NP2}\\
(a\ast_i b)\diamond_i c
&=\varepsilon_i a\ast_i(b\diamond_i c),
\label{eq:tensor-unified-NP3}\\
\delta(a\diamond_i b)\ast_i c
-\varepsilon_i a\diamond_i(b\ast_i c)
&=
\varepsilon_i\delta(b\diamond_i a)\ast_i c
-b\diamond_i(a\ast_i c).
\label{eq:tensor-unified-NP4}
\end{align}
For $\varepsilon_i=-1$, these are precisely
\eqref{deltaMN2}, \eqref{deltaMN1},
\eqref{id: NP  ass}, and \eqref{mockNP}, respectively; for
$\varepsilon_i=1$, they are the corresponding
$\delta$-Novikov-Poisson identities.

The point is that the signs in the two tensor factors combine
multiplicatively. For example, iterating~\eqref{tensor-product} gives
\[
\begin{aligned}
(x\diamond y)\diamond z
={}&((x_1\diamond_1y_1)\diamond_1z_1)
   \otimes((x_2\ast_2y_2)\ast_2z_2)\\
&+((x_1\diamond_1y_1)\ast_1z_1)
   \otimes((x_2\ast_2y_2)\diamond_2z_2)\\
&+((x_1\ast_1y_1)\diamond_1z_1)
   \otimes((x_2\diamond_2y_2)\ast_2z_2)\\
&+((x_1\ast_1y_1)\ast_1z_1)
   \otimes((x_2\diamond_2y_2)\diamond_2z_2).
\end{aligned}
\]
Using~\eqref{eq:tensor-unified-NP1} and
\eqref{eq:tensor-unified-NP3}, together with the two sign relations
for $\ast_i$, transforms the right-hand side into
$\varepsilon(x\diamond z)\diamond y$.

The first compatibility condition illustrates the same cancellation
more directly:
\[
\begin{aligned}
(x\ast y)\diamond z
={}&((x_1\ast_1y_1)\diamond_1z_1)
   \otimes((x_2\ast_2y_2)\ast_2z_2)\\
&+((x_1\ast_1y_1)\ast_1z_1)
   \otimes((x_2\ast_2y_2)\diamond_2z_2)\\
={}&\varepsilon\,x\ast(y\diamond z),
\end{aligned}
\]
where~\eqref{eq:tensor-unified-NP3} is used in both factors.
Grouping the four terms in the corresponding expansions and applying
\eqref{eq:tensor-unified-NP2} and
\eqref{eq:tensor-unified-NP4}, respectively, gives
\[
\begin{aligned}
\delta(x\diamond y)\diamond z
-\varepsilon x\diamond(y\diamond z)
&=
\varepsilon\delta(y\diamond x)\diamond z
-y\diamond(x\diamond z),\\
\delta(x\diamond y)\ast z
-\varepsilon x\diamond(y\ast z)
&=
\varepsilon\delta(y\diamond x)\ast z
-y\diamond(x\ast z).
\end{aligned}
\]
Hence the tensor product is a $\delta$-Novikov-Poisson algebra when
$\varepsilon=1$, and a $\delta$-mock-Novikov-Poisson algebra when
$\varepsilon=-1$.

It remains to consider the pairs
\[
(\mathcal P_\delta,\mathrm{m}\mathcal P_\delta)
\qquad\text{and}\qquad
(\mathrm t\mathcal P_\delta,\mathrm{tm}\mathcal P_\delta).
\]
Here $\diamond_i=\{-,-\}_i$, and the symmetry of the bracket is
uniformly expressed by
$\{a,b\}_i=-\varepsilon_i\{b,a\}_i$.
Consequently,
$\{x,y\}=-\varepsilon\{y,x\}$ on
$\mathcal A_1\otimes\mathcal A_2$.

We next verify the Jacobi-type identity. Since both Lie and mock-Lie
brackets satisfy the same cyclic Jacobi identity, the pure terms in
the cyclic expansion of
$\sum_{\rm cyc}\{\{x,y\},z\}$ vanish. Indeed,
$(a\ast_i b)\ast_i c$ is invariant under cyclic permutations.
The remaining mixed terms are
\[
\begin{aligned}
M={}&
\sum_{\rm cyc}
(\{x_1,y_1\}_1\ast_1z_1)
 \otimes\{x_2\ast_2y_2,z_2\}_2\\
&+
\sum_{\rm cyc}
\{x_1\ast_1y_1,z_1\}_1
 \otimes(\{x_2,y_2\}_2\ast_2z_2).
\end{aligned}
\]

Suppose first that
$(\mathsf X_\delta,\mathsf Y_\delta)
=(\mathcal P_\delta,\mathrm{m}\mathcal P_\delta)$.
The two compatibility identities are simultaneously written as
\[
\{a\ast_i b,c\}_i
=
\varepsilon_i\delta
\bigl(
a\ast_i\{b,c\}_i+\{a,c\}_i\ast_i b
\bigr).
\]
Substituting this identity into the two cyclic sums defining $M$ and
using
$a\ast_i b=\varepsilon_i b\ast_i a$ and
$\{a,b\}_i=-\varepsilon_i\{b,a\}_i$,
the mixed terms cancel pairwise after cyclic reindexing. Hence
$M=0$, so the induced bracket is Lie if $\varepsilon=1$ and
mock-Lie if $\varepsilon=-1$. Applying the same identity directly to
\eqref{tensor-product} yields
\[
\{x\ast y,z\}
=
\varepsilon\delta
\bigl(
x\ast\{y,z\}+\{x,z\}\ast y
\bigr),
\]
which is the required compatibility condition.

Finally, let
$(\mathsf X_\delta,\mathsf Y_\delta)
=(\mathrm t\mathcal P_\delta,\mathrm{tm}\mathcal P_\delta)$.
In this case the defining compatibility conditions take the uniform
form
\[
2\delta\varepsilon_i\,a\ast_i\{b,c\}_i
=
\{a\ast_i b,c\}_i+\{b,a\ast_i c\}_i.
\]
Taking the cyclic sum and using the symmetry rules for $\ast_i$ and
$\{-,-\}_i$ gives
\[
a\ast_i\{b,c\}_i+
b\ast_i\{c,a\}_i+
c\ast_i\{a,b\}_i=0.
\]
These two identities applied to the mixed expression $M$ again give
$M=0$. Moreover, expanding the transposed compatibility expression
on the tensor product and applying the factor identities gives
\[
2\delta\varepsilon\,x\ast\{y,z\}
=
\{x\ast y,z\}+\{y,x\ast z\}.
\]
Thus the tensor product is a transposed $\delta$-Poisson algebra for
$\varepsilon=1$ and a transposed $\delta$-mock-Poisson algebra for
$\varepsilon=-1$.

Since $\varepsilon=\varepsilon_1\varepsilon_2$, the conclusion follows
in all three cases.
\end{proof}

\begin{remark}
The rules of
Theorem~\ref{thm:tensor-unified} do not extend in general to
$\mathrm{m}\mathcal N_\delta$ or
$\mathrm{m}\mathcal{GD}_\delta$. Nevertheless, the theorem yields the
following related tensor construction for $\delta$-mock-Novikov
algebras.
\end{remark}

\begin{corollary}\label{cor:mock-tensor-Novikov}
Let $(\mathcal{A},\circ)$ be a $\delta$-mock-Novikov algebra and let
$(\mathcal{B},\ast)$ be an anti-commutative anti-associative algebra.
Define
\[
(a\otimes p)\bullet(b\otimes q)
=(a\circ b)\otimes(p\ast q).
\]
Then $(\mathcal{A}\otimes\mathcal{B},\bullet)$ is a
$\delta$-Novikov algebra.
\end{corollary}

\begin{proof}
Regard $(\mathcal A,0,\circ)$ and $(\mathcal B,\ast,0)$ as
$\delta$-mock-Novikov-Poisson algebras and apply
Theorem~\ref{thm:tensor-unified}.
\end{proof}





\end{document}